\documentclass[a4paper]{amsart}

\usepackage{amsfonts,amsmath, amssymb, latexsym,amsrefs,appendix}
\usepackage{amsthm}
\usepackage{amscd}
\usepackage[all,cmtip]{xy}
\usepackage{float}
\usepackage{enumerate}
\usepackage{tikz}
\usetikzlibrary{arrows}
\newcommand{\midarrow}{\tikz \draw[-triangle 90] (0,0) -- +(.03,0);}

\newcommand{\cat}{\mathbf}

\newcommand{\cS}{{\bf S}}

\newcommand{\beq}{\begin{eqnarray}}
\newcommand{\eeq}{\end{eqnarray}}
\newcommand{\bcen}{\begin{center}}
\newcommand{\ecen}{\end{center}}
\newcommand{\ff}{\infty}
\newcommand{\De}{\Delta}

\newtheorem{thm}{Theorem}[section]
\newtheorem{cor}[thm]{Corollary}
\newtheorem{lem}[thm]{Lemma}
\newtheorem{prop}[thm]{Proposition}
\newtheorem{quest}[thm]{Question}

\theoremstyle{remark}
\newtheorem{rem}[thm]{Remark}

\theoremstyle{definition}
\newtheorem{defn}[thm]{Definition}

\theoremstyle{plain}
\newtheorem{theoremoutside}[thm]{Theorem}

\newtheorem{introtheorem}{Theorem}

\newtheorem{propout}[thm]{Proposition}

\theoremstyle{remark}
\newtheorem*{ack}{Acknowledgements}

\theoremstyle{remark}
\newtheorem*{nota}{Notation and Conventions}

\theoremstyle{remark}

\theoremstyle{plain}

\def\eq#1/{(\ref{e:#1})}
\def\Section#1/{Section~\ref{s:#1}}
\def\Table#1/{Table~\ref{t:#1}}
\def\Figure#1/{Figure~\ref{f:#1}}

\newcommand{\pt}{\{ \ast \}}

\newcommand{\NN}{{\mathbb N}}

\newcommand{\CC}{{\mathbb C}}
\newcommand{\ZZ}{{\mathbb Z}}
\newcommand{\QQ}{{\mathbb Q}}
\newcommand{\FF}{{\mathbb F}}
\newcommand{\LL}{{\mathbb L}}

\newcommand{\co}{\colon\thinspace}

\newcommand{\rank}{\mathrm{rank}}
\newcommand{\ke}{\mathrm{ker}}

\newcommand{\pcp}{{^{\wedge}_p}}
\newcommand{\ratcp}{{^{\wedge}_{\mathbb Q}}}

\newcommand{\Fq}{{{\mathbb F}_{q}}}
\newcommand{\Fp}{{{\mathbb F}_{p}}}
\newcommand{\Fpk}{{{\mathbb F}_{p^k}}}

\newcommand{\qcp}{{^{\wedge}_q}}

\newcommand{\BKfbar}{BK({\overline{{\mathbb F}}}_p)}
\newcommand{\Kfbar}{K({\overline{{\mathbb F}}}_p)}

\newcommand{\Fbar}{{{\overline{\mathbb F}_{p}}}}

\newcommand{\BGfq}{{BG({\mathbb F}_{p^k})^{\wedge}_q}}
\newcommand{\BKfq}{{BK({\mathbb F}_{p^k})^{\wedge}_q}}
\newcommand{\BKf}{{BK({\mathbb F}_{p^k})}}
\newcommand{\BGfqbar}{{BG({\overline{\mathbb F}}_p)^{\wedge}_q}}

\newcommand{\BGIfbar}{{BG_I({\overline{\mathbb F}}_p)}}
\newcommand{\BGJfbar}{{BG_J({\overline{\mathbb F}}_p)}}
\newcommand{\BGIfbarq}{{BG_I({\overline{\mathbb F}}_p)^{\wedge}_q}}

\newcommand{\BGq}{{BG^{\wedge}_q}}

\newcommand{\BKq}{{BK^{\wedge}_q}}

\newcommand{\BKqfix}{{(BK^{\wedge}_q)^{h\psi^k}}}
\newcommand{\BGIq}{{{BG_I}^{\wedge}_q}}
\newcommand{\BGIfq}{{{BG_I({\mathbb F}_{p^k})}^{\wedge}_q}}

\newcommand{\Etwo}{E_2^{\ast, \ast}}

\newcommand{\BGIqfix}{{({{BG_I}^{\wedge}_q})^{h\psi^k_I}}}

\newcommand{\BGI}{{BG_I}}

\newcommand{\BPI}{{BP_I}}

\newcommand{\BTqfix}{{({{BT}^{\wedge}_q})^{h\psi^k}}}
\newcommand{\BGoneqfix}{{({{BG_1}^{\wedge}_q})^{h\psi^k}}}
\newcommand{\BGtwoqfix}{{({{BG_2}^{\wedge}_q})^{h\psi^k}}}

\newcommand{\HBKq}{H^{\ast}(BK, \Fq)}
\newcommand{\HBGevenq}{H^{2\ast}(BK, \Fq)}
\newcommand{\HBGoddq}{H^{2\ast+1}(BK, \Fq)}
\newcommand{\HBTq}{H^{\ast}(BT, \Fq)}
\newcommand{\HBGIq}{H^{\ast}(BG_I, \Fq)}
\newcommand{\HBGoneq}{H^{\ast}(BG_1, \Fq)}
\newcommand{\HBGtwoq}{H^{\ast}(BG_2, \Fq)}

\newcommand{\HBKqfix}{H^{\ast}({{(BK^{\wedge}_q)^{h\psi^k}}}, \Fq)}
\newcommand{\HBGIqfix}{H^{\ast}({{({BG_I}^{\wedge}_q)^{h\psi^k}}}, \Fq)}
\newcommand{\HBGoneqfix}{H^{\ast}({{({{BG_1}^{\wedge}_q})^{h\psi^k}}}, \Fq)}
\newcommand{\HBGtwoqfix}{H^{\ast}({{({{BG_2}^{\wedge}_q})^{h\psi^k}}}, \Fq)}

\newcommand{\HBTqfix}{H^{\ast}({({BT^{\wedge}_q})}^{h\psi^k}, \Fq)}

\newcommand{\HBKqfixhoco}{H^{\ast}(\hcl_{I\in \cS} {(\BGIq)}^{h\psk}, \Fq)}

\newcommand{\Tor}{Tor^{\ast, \ast}}

\newcommand{\cR}{\mathbf{R}}

\newcommand{\hcl}{\mathrm{hocolim}}
\newcommand{\cl}{\mathrm{colim}}
\newcommand{\Objects}{\mathrm{Objects}}
\newcommand{\Hom}{\mathrm{Hom}}
\newcommand{\modd}{\mathrm{mod}}
\newcommand{\Stab}{\mathrm{Stab}}
\newcommand{\Witt}{\mathrm{Witt}}
\newcommand{\hl}{\mathrm{holim}}

\newcommand{\nod}{\noindent}

\newcommand{\sset}{{\subseteq}}

\newcommand{\psk}{\psi^k}

\newcommand{\ari}{\ar@{^{(}->}}
\newcommand{\arif}{\ar@{_{(}->}}
\newcommand{\arline}{\ar@{-}}
\newcommand{\arele}{\ar@{|->}}

\begin{document}

\title[Discrete approximations for Kac-Moody groups]{Discrete approximations for complex  Kac-Moody groups}
\author[J. D. Foley]{John D. Foley}
\address{Department of Mathematics, University of Copenhagen, 2100 K{\o}benhavn {\O}, Denmark}
\thanks{Support by the Danish National Research Foundation (DNRF) through the Centre for Symmetry and Deformation (DNRF92) during the preparation of this paper is gratefully acknowledged.}
\email{foley@math.ku.dk}


\begin{abstract}
We construct a map from the classifying space of
a discrete Kac-Moody group over the algebraic closure of the field with $p$ elements to the classifying space of a complex topological Kac-Moody group
and prove that it is a homology equivalence at primes $q$ different from $p$.
This generalises a classical result of Quillen--Friedlander--Mislin for Lie groups.
  As an application, we construct unstable Adams operations for general Kac-Moody groups compatible with the Frobenius homomorphism.
Our results rely on new integral homology decompositions for certain infinite dimensional unipotent  subgroups of  discrete Kac-Moody groups.
\end{abstract}


\maketitle


\section{Introduction}

Cohomological approximations for Lie groups by related discrete groups were developed  by Quillen \cites{QKT}, Milnor \cites{Milnor}, Friedlander and Mislin \cite{Fbar}.
In this paper, we prove 
 that a complex Kac-Moody group is cohomologically approximated by the corresponding discrete Kac-Moody group over $\Fbar$ at primes $q$ different from $p$.
 One application is the construction of unstable Adams operations for arbitrary complex Kac-Moody groups.

Over the complex numbers, topological Kac-Moody groups are constructed by integrating Kac-Moody Lie algebras \cite{kacintegrate} that are typically infinite dimensional, but integrate to Lie groups
when finite dimensional.
   Kac-Moody Lie algebras are defined via generators and relations encoded in a  generalized Cartan matrix 
   \cites{kumar}.
 Kac-Moody groups $K$ have a finite rank maximal torus that is unique up to conjugation by a Coxeter Weyl group; this Weyl group is finite exactly when $K$ is Lie.
For any minimal split topological $K$ \cites{kptop, kumar},
 Tits \cite{TitsKM} constructed a corresponding  discrete Kac-Moody group functor $K(-)$ from the category of commutative rings with unit to the category of groups
 such that $K=K(\CC)$ as abstract groups.
  In this paper, a Kac-Moody group is a value of such a functor applied to some fixed ring or the corresponding connected topological group \cites{kptop, kumar}.

\begin{introtheorem}
Let $K$ be a topological complex Kac-Moody group and let $\Kfbar$ be the value of the corresponding discrete Kac-Moody group functor. 
Then, there exists a map $\BKfbar \rightarrow BK$ that is an isomorphism
on homology with $\Fq$ coefficients for any $q\neq p$. 
\label{KMbar}
\end{introtheorem}

Theorem \ref{KMbar} is  proved  by extending Friedlander and Mislin's map \cite[Theorem 1.4]{Fbar}
between the classifying spaces of appropriate reductive subgroups of discrete and topological Kac-Moody groups to a map between the full classifying spaces of
Kac-Moody groups. This extension uses a new homology  decomposition of a discrete Kac-Moody group over a field away from the ambient characteristic.
We state this result now, but see Theorem \ref{hocoBGqfinite} for a more precise statement.

\begin{introtheorem}
Let $K(\FF)$  be a Kac-Moody group over a field.
 Then there is a finite collection of subgroups $\{ G_I(\FF)\}_{I\in \cS}$ that are the $\FF$--points of reductive algebraic groups
such that the inclusions $ G_I(\FF) \hookrightarrow K(\FF)$ induce a homology equivalence
\beq
\hcl_{I\in \cS}\BGI(\FF) {\longrightarrow} BK(\FF),
\label{introhocoBGq}
\eeq
\nod away from the characteristic of $\FF$.
\label{introhocoBGqfinite}
\end{introtheorem}
\nod This gives a natural way to propagate cohomological approximations of complex reductive Lie groups to Kac-Moody groups; see also Remark \ref{FMc}.

Theorem \ref{introhocoBGqfinite} in turn depends on a homological vanishing result 
 for key infinite dimensional unipotent subgroups of discrete Kac-Moody groups over fields.
 As explained in \ref{sec:BN}--\ref{sec:rgd}, these subgroups play the same role in the subgroup combinatorics of $K(\FF)$ as the unipotent radicals of parabolic subgroups play in that
 of an algebraic group.

 \begin{introtheorem}
Let $U_I(\FF)$ be the unipotent factor of a parabolic subgroup of a discrete Kac-Moody group over a field. Then,
$H_n(BU_I(\FF), \LL)=0$ for all $n>0$ and any field
$\LL$ of different characteristic.
\label{vanishthm}
\end{introtheorem}

To prove Theorem \ref{vanishthm}, we obtain new structural understanding of these unipotent subgroups---e.g., we obtain new colimit presentations---based in the geometric group theory of Kac-Moody groups but developed on the level of classifying spaces.
 Theorem \ref{hocoUnewthmslick} decomposes the classifying spaces of
  key unipotent subgroups
   as  homotopy colimits
   of finite dimensional unipotent subgroup classifying spaces. These new integral
   homotopy decompositions are natural with respect to Tits's  construction of $K(-)$ \cite{TitsKM} and apply to arbitrary groups with root group data systems (see \ref{sec:rgd}).
    We use a functorial comparison
    of diagrams indexed by the Weyl group to  diagrams indexed by its Davis complex \cite{Davis} (Theorem \ref{combin}) that may be of interest to Coxeter group theorists.  This comparison ultimately reduces our homology decompositions to the contractibility of the Tits building.  
The structural understanding developed also provides
a method to compute the non-trival cohomology of these unipotent subgroups over a field at its characteristic (see Theorem \ref{treecalc}).

Our main application of Theorem \ref{KMbar} is the construction of unstable Adams operations.  These maps are defined as $\psi$ that
fit into the homotopy commutative diagram
\beq
\xymatrixcolsep{4pc}\xymatrix{
BT \ar[r]^{B(t \mapsto t^p)} \ar[d] &  BT  \ar[d]\\
BK  \ar[r]^{\psi} & BK
}
\label{Adamsdef}
\eeq
\nod where the vertical maps are induced by the inclusion $T \le K$ of the standard maximal torus. 
We are more generally interested in $q$--local unstable Adams operations where (\ref{Adamsdef}) is replaced by its functorial localization $(-)\qcp$ with respect to $\Fq$--homology \cite[1.E.4]{Fploc}.
For connected Lie  groups, unstable Adams operations (including $q$--local versions) are  unique, up to homotopy, whenever they exist \cite{JMOconn}.
In \cite{BKtoBK}, unstable Adams operations for simply connected rank 2 Kac-Moody groups where constructed and shown to be unique. 

Classically \cite{Fbar}, Theorem \ref{KMbar} was shown for the group functor $G(-)$ associated to a complex reductive Lie group $G$.
For such $G$,
 the $p$th unstable Adams operation on $\BGq$ is homotopic to
the self-map $B(G(\varphi))\qcp$ for $\varphi$ the Frobenius homomorphism  $x \mapsto x^p$ on $\Fbar$. 
We use the functoriality  of Tits's construction to show that $B(K(\varphi))\qcp$ is an unstable Adams operation for an arbitrary topological Kac-Moody group
via the comparison map of Theorem \ref{KMbar}. The construction of this map depends on choices as in \cite{Fbar} (see \ref{sec:compare}). We
 do not resolve the question of uniqueness.

\begin{introtheorem}
Let $K$  be a topological Kac-Moody group over $\CC$ with Weyl group $W$.
The Frobenius map induces  $p$th $q$--local unstable Adams operations $\psi\co BK\qcp {\rightarrow}  BK\qcp$ for any prime $q \neq p$. 
If there is no element of order $p$ in $W$, these local unstable Adams operations can be assembled into a global unstable Adams operation $\psi: BK {\rightarrow}  BK$.
\label{pskthm}
\end{introtheorem}

For reductive $G$,
homotopy fixed points $(\BGfqbar)^{h\varphi^k}$ with respect to the map induced by the Frobenius homomorphism coincide   with ordinary fixed points  $\BGfq$ \cites{Friedbook,Fbar}.
Our work began to investigate if this fact could be extended to Kac-Moody groups. 
Among discrete Kac-Moody groups,  $K(\Fpk)$  are of special interest  as a class of finitely generated, locally finite groups (cf. \cite{propT});
 R\`{e}my develops an extended argument  in \cite{KMasdiscrete} that $K(\Fpk)$ should be considered generalizations of certain $S$--arithmetic groups.
 Like $G(\Fpk)$ classically, understanding the cohomology and homotopy of $K(\Fpk)$ via the better understood $K$ (cf. \cite{nitutkm}) is desirable.

 From the perspective of homotopical group theory, unstable Adams operations on $p$--compact groups yield examples of $p$--local finite groups after taking homotopy fixed points \cite{BrMø} (see \cite{AG} similar
 results for $p$--local compact groups).
To the extent that homotopy Kac-Moody groups may generalize homotopy Lie groups (as proposed e.g., in Grodal's 2010 ICM address \cite{JG}), unstable Adams operations on Kac-Moody groups and their homotopy fixed points
are interesting.
Through cohomology calculations with classifying spaces of compact Lie subgroups, we provide evidence that, in contrast to the classical setting, $(\BKq)^{h\psk}$ and $\BKfq$ rarely agree.

\begin{introtheorem}
Let $K$ be  an infinite dimensional simply connected complex Kac-Moody group of rank 2 and $\psi$ be its unique $p$th $q$--local unstable Adams operation
for primes $q \neq p$  with $q$ odd. Then,  $H^\ast((\BKq)^{h\psk}, \Fq)=H^\ast(\BKf,\Fq)$  if and only if they both vanish.
\label{rank2thm}
\end{introtheorem}
\nod
Our comparison methodology for Theorem \ref{rank2thm} is applicable to general Kac-Moody groups, but there are many technical challenges (see the discussion beginning \ref{subsec:r2})  and possible issues with uniqueness.
Nevertheless,   Theorem \ref{introhocoBGqfinite} does provide a homology decomposition of $K(\Fpk)$ in terms of finite subgroups.

\subsection{Organization of the paper}
As described above, the logical progression of this paper begins with a structural understanding of
key unipotent subgroups of discrete Kac-moody groups. This structure implies the vanishing theorem \ref{vanishthm} which
leads to a proof of Theorem \ref{introhocoBGqfinite}.
 Our map from $\BKfbar$  to $BK$ in Theorem \ref{KMbar} is constructed by  a compatible
family of Friedlander--Mislin maps \cite[Theorem 1.4]{Fbar} which give the desired homological approximation by Theorem \ref{introhocoBGqfinite}.
The functoriality  of Tits's $\Kfbar$ induces an unstable Adams operation on $BK$, but cohomology calculations for Theorem \ref{rank2thm} show that
we cannot expect $\BKfq$ to be expressible as the homotopy fixed points of these operations.

Concretely, this paper is structured as follows.
In Section \ref{sec:combo},  we collect  background material from geometric group theory 
and methods to manipulate homotopy colimits. 
 Section  \ref{sec:homodecomp} proves our
 unipotent subgroup homotopy decompositions, Theorem \ref{hocoUnewthmslick},  based in subgroup combinatorics for root group data systems
 and a diagram comparison Theorem \ref{combin}
 based in Coxeter combinatorics.
In \ref{sec:van}, we describe applications of Theorem \ref{hocoUnewthmslick} to discrete Kac-Moody groups over a field, 
including proofs of Theorems \ref{KMbar} and \ref{introhocoBGqfinite}.
We employ these results in \ref{sec:compare} to define unstable Adams operations for general Kac-Moody groups as in Theorem \ref{pskthm}.
Section \ref{sec:compmain} proves Theorem \ref{rank2thm} and
outlines explicit calculations of the group cohomology of the main unipotent subgroup of $K(\Fpk)$  at the prime $p$ in most cases where the Weyl group is a free product.

 \begin{nota}
 As noted above, $K(R)$, for a commutative ring $R$,  denotes the value of one of Tits's  explicit
  discrete, minimal, split  Kac-Moody group functors \cite[3.6]{TitsKM}. We generally use $K$ for a minimal topological Kac-Moody group \cite[7.4.14]{kumar}, but to simply typography, we occasionally
  abbreviate  $K(R)$ to $K$ when $R$ is fixed throughout an argument.
 We define the rank of $K$ to be the rank of its maximal torus. A  generalized Cartan matrix is  a square integral matrix $A=(a_{ij})_{1 \le i,j \le n}$ such that
$a_{ii}=2 $, $a_{ij} \le 0$ and  $a_{ij}=0 \iff a_{ji}=0$ and we reserve the notation $A$ for the generalized Cartan matrix used in constructing $K$.  To simplify statements, we use $(-)\qcp$ to denote a Bousfield $\Fq$--homology localization functor of spaces \cite[1.E.4]{Fploc} rather than the more common Bousfield-Kan $q$-completion functor \cite{BK}. These two functors agree, up to homotopy, for any space $X$ that is nilpotent or has $H_1(X,\Fq)=0$. Math mode bold is used for categories.  For instance, each Coxeter group $W$ has an associated poset $\cat{W}$.  Our results are stated in category of topological spaces with the usual weak equivalences $\cat{Top}$, but
$\cat{Spaces}$---which can be taken to be topological
spaces, $CW$--complexes, or simplicial sets---is used in technical results when we wish emphasize that there is little dependence on $\cat{Top}$.
  \end{nota}

  \begin{ack}
  The results of this paper come from the author's thesis \cite{foleythesis}.  The author is
   most grateful to Nitu Kitchloo who advised this thesis, introduced the author to Kac-Moody groups, and suggested looking for a generalization of the work of Quillen--Friedlander--Mislin.
 The author is also indebted to peers with whom we discussed this work, especially Ben Hummon, Joel Dodge and Peter Overholser.  Some of the results of this paper were proved independently in rank two cases \cite{AR} by explicit cohomology calculations and descriptions of unipotent subgroups.
   More specifically, Proposition 4.3 and Section 9 of \cite{AR} describe specializations of our Theorems \ref{vanishthm} and \ref{rank2thm}, respectively and \cite[Proposition 8.2]{AR} overlaps with our Table \ref{colimoffixed}.
  The support of the Danish National Research Foundation (DNRF) through the Centre for Symmetry and Deformation (DNRF92) during the preparation of this paper is gratefully acknowledged.
  \end{ack}

\section{Combinatorial tool kit}
\label{sec:combo}

Like Lie groups, Kac-Moody groups have (typically infinite) Weyl groups that underlie combinatorial structures on subgroups including  $BN$--pairs and root group data systems.
In \ref{sec:cox}--\ref{sec:rgd}, the aspects of these structures we use are assembled. We also collect tools for manipulating homotopy colimits (\ref{sec:pbmain}--\ref{sec:trans}) which relate to colimit presentations of groups via Seifert-van Kampen theory (\ref{sec:cosetgeo}). With a homotopy theoretic framework in place, \ref{sec:known} quickly reviews two known homotopy decompositions associated to Coxeter groups and $BN$--pairs
and how they apply to Kac-moody groups.

\subsection{Coxeter groups}
\label{sec:cox}

Because of their role in $BN$--pairs and RGD systems, the structure of Coxeter groups will be important to the arguments here.  For our purposes, a Coxeter group will be a {\em finitely} generated group with a presentation
  \beq
  \langle s_1, s_2, \ldots s_n | s_i^2=1,  (s_i s_j)^{m_{ij}} = 1 \rangle
  \label{cox}
  \eeq
  \nod for $1 \le i, j \le n$ and fixed $2 \le m_{ij}=m_{ji} \le \ff $ with  $m_{ij}=\ff$ specifying a vacuous relation.  For example, the presentation the Weyl group of a Kac-Moody group is determined by its generalized Cartan matrix
  $A=(a_{ij})_{1 \le i,j \le n}$, i.e.
  \beq
  m_{ij}=\left\{
    \begin{array}{lr}
    2 & ~ a_{ji}a_{ij}=0\\
   3 & ~ a_{ji}a_{ij}=1\\
   4 & ~ a_{ji}a_{ij}=2\\
   6 & ~ a_{ji}a_{ij}=3\\
   \ff & ~ a_{ji}a_{ij}\ge 4
    \end{array}
\right.
\label{kmweyl}
  \eeq
  \nod for all $i\neq j$ \cite[p. 25]{kumar}. 
The solution to the word problem for Coxeter groups gives us a detailed picture
of $W$.

\begin{theoremoutside}[Word Problem \cite{Buildings}] Let $W$ be a Coxter group.  For any words ${\vec{v}}$ and ${\vec{w}}$ in letters   $s_1, s_2, \ldots s_n$
\begin{itemize}

\item A word ${\vec{w}}$ is reduced in $W$ if and only if it cannot be shortened by a finite sequence of $s_i s_i$ deletions and replacing of the length $m_{ij}< \ff$ alternating words $s_i s_j \ldots$ by  $s_j s_i \ldots$ or vice versa.

\item  Reduced words ${\vec{v}}$ and  ${\vec{w}}$ are equal in $W$ if and only if ${\vec{v}}$ transforms into ${\vec{w}}$ via a finite sequence of length $m_{ij}< \ff$ alternating word by  replacements as describe above.
\end{itemize}
\label{wordproblem}
\end{theoremoutside}
Thus, every element of a Coxeter group  has a well-defined  reduced word length. This provides a poset structure on $W$, called the {\bf weak Bruhat order}, defined as
\beq
v \le w \iff w = vx ~~~{\rm and}~~~ l(w)=l(v)+l(x)
\label{weakorder}
\eeq
\nod where $x,v,w \in W$ and $l(y)$ is the reduced word length of $y$.
This poset will be important to us because it dictates the intersection pattern of the subgroups $U_w$ (\ref{uwdef}) of a group with RGD system (see \ref{sec:rgd}).

For any $I \sset S$ define
$
W_I := \langle \{ s_I\}_{i \in I \sset S}  \rangle
$ and
let
\beq
\cS := \{ I \sset S | |W_I| < \ff  \}
\label{Sposet}
\eeq
\nod be the poset ordered by inclusion.  Note that
 $I \in \cS$ implies that $wW_I$ has a unique longest word $w_I$.  We call $I \in \cS$ and their associated $W_I$  {\bf finite type}.
We emphasize that a Kac-Moody group is of Lie type if and only if its Weyl group given by (\ref{kmweyl}) is finite.

\subsection{$BN$--pairs}
\label{sec:BN}

Here we follow {\em Buildings} \cite{Buildings}. 

\begin{defn}
A group $G$ is a $BN$--pair if it has data $(G, B, N, S)$ where  $B$ and $N$ are distinguished normal subgroups that together generate $G$.  Furthermore, $T := B \cap N \unlhd N$ and $W := N/T$ is generated by $S$ so that for $\overline{s} \in s \in S$ and $\overline{w} \in w \in W$:

{\bf BN1:} $\overline{s} B \overline{w} \subset B \overline{sw} B \cup B \overline{w} B$

{\bf BN2:} $\overline{s} B \overline{s^{-1}} \nsubseteq B$ .

\end{defn}

Note that the above sets are well-defined as $B \geq T  \unlhd N$.  It is common in the literature, and will be common in this paper, to drop the bars and avoid reference to particular representatives of elements of $W$.

  For Kac-Moody groups,  $T$  is the standard maximal torus of rank $2n-\rank(A)$, $N$ is the normalizer of $T$,  $|S|=n$ is the size of the generalized Cartan matrix and $B$ is the standard Borel subgroup defined analogously to the Lie case.  In line with  \ref{sec:cox}, we will further require that $S$ is a finite set.

Important properties of $BN$--pairs include that $W$ is a Coxeter group and $G$ admits a {\bf Bruhat decomposition}
\beq
G = \coprod_{w\in W} B w B
\label{bruhatdecomp1}
\nonumber
\eeq
\nod so that all subgroups of $G$ containing $B$ are of the form
$
P_I = \coprod_{w\in W_I} B w B
$ where $I \sset S$ and $W_I$ is  generated by $I$.  These subgroups are called {\bf standard parabolic subgroups} and inherit the structure of a $BN$--pair with data $(P_I, B, N \cap P_I, I)$ and Weyl group $W_I$.
For Kac-Moody groups $A_J:=(a_{ij})_{i,j \in J}$ of $A$ determines a Weyl group  $W_J$ group for $P_J$ via (\ref{kmweyl}).
  The $BN$--pair axioms
then lead \cite{kumar} to {\bf generalized Bruhat decompositions}
\beq
G  = \coprod_{w\in W_J \setminus W / W_I} P_J w P_I .
\label{bruhatdecomp}
\eeq

\subsection{Root group data systems}
\label{sec:rgd}

Much of the work in this paper could be adapted to refined Tits systems \cite{relations} but we prefer to work with root group data (RGD) systems. Notably the framework of RGD systems has been used to prove \cite{CR, Buildings} the colimit presentation induced by  Theorem \ref{hocoUnewthmslick} (see Remark \ref{rem:cosetgeo} in \ref{sec:cosetgeo}) which was conjectured by Kac and Petersen in \cite{relations}.

Briefly, a RGD system for a group $G$ is a given by a tuple $(G,\{ U_\alpha \}_{\alpha \in \Phi}, T)$ for $\Phi$  associated to a Coxeter system $(W,S)$.
The elements of the set $\Phi$ are called roots and  $U_\alpha$ are nontrivial subgroups of $G$ known as {\bf root subgroups}.  The RGD subgroups generate $G$, i.e. $G = \langle T, \{ U_\alpha \}_{\alpha \in \Phi} \rangle$.
  In the case of Tits's Kac-Moody group functor,
 $U_\alpha$ are isomorphic to the base ring as a group under addition and $T$ is the standard maximal torus isomorphic to a finite direct product of the multiplicative group of units of the base ring. As the complete definition is somewhat involved and the arguments here can be followed with the Kac-Moody example in mind, we refer the reader to \cite{Buildings} for the standard definition of a RGD system and further details
  but note that \cite{CR} provides an alternative formulation.

The RGD structure regulates  the conjugation action of $W$ on $U_\alpha$.
 In particular, $\Phi$ has a  $W$--action
 and is the union of the orbits of the simple roots, $\alpha_i$, corresponding to the elements, $s_i$, of $S$. For $n \in w \in W$
$
 U_{w\alpha}= nU_\alpha n^{-1}
 $.

  In the Kac-Moody case, $\Phi = W \{\alpha_i\}_{1\le i\le n} \sset \sum_{1\le i\le n} \ZZ \alpha_i$ with the $W$--action given by
  \beq s_j (\sum_{i=1}^{k} n_i \alpha_i) = \sum_{1 \le i \neq j \le k} n_i \alpha_i - ( n_j + \sum_{1 \le i \neq j \le k} n_i a_{ij})\alpha_j  .
\label{actiondef2}
\eeq
\nod where $a_{ij}$ are entries in the generalized Cartan matrix.

The set of roots $\Phi$  is divided into positive and negative roots so that $\Phi= \Phi^- \coprod \Phi^+$. For Kac-Moody groups,
$\Phi^-$ and $\Phi^+$
corresponds to the subsets of $\Phi$ with all negative or all positive coefficients, respectively.  Define
\beq
U^\pm =  \langle U_\alpha  ; \alpha \in \Phi^\pm \rangle.
\label{U+def}
\eeq
\nod The group $G$ carries the structure of a $BN$--pair for either choice of $B = TU^\pm$ and $N$ the normalizer of $T$.
There are also well-defined
\beq
 U_w = \langle U_{\alpha_{i_1}},  U_{s_{i_1}\alpha_{i_2}}, \ldots  U_{s_{i_1} \ldots s_{i_{k-1}}\alpha_{i_k}} \rangle
 \label{uwdef}
 \eeq
\nod  where $\Theta_w:=\{ {\alpha_{i_1}}, \alpha_{i_2}, \ldots, {ws_{i_l}\alpha_{i_l} }\}=\Phi^{+} \cap w\Phi^{-}$ is well-defined for any reduced expression  for $w=s_{i_1} \ldots s_{i_k}$.  Moreover, the multiplication map
\beq
U_{\alpha_{i_1}} \times U_{s_{i_1}\alpha_{i_2}} \times \ldots \times U_{ws_{i_l}\alpha_{i_l} } \stackrel{m}{\longrightarrow} U_w
\label{multiply}
\eeq
\nod is an isomorphism of sets for any choice of reduced expression  for $w=s_{i_1} \ldots s_{i_k}$.
 For any index set $I \sset W$
\beq
\bigcap_I U_w = U_{\inf_\cat{W} \{I\}}
\label{Uwintersect}
\eeq
\nod where the greatest lower bound is with respect to the weak Bruhat order (\ref{weakorder}).

For a group with RGD system, there is a symmetry between the positive and negative roots, as in the Kac-Moody case. In particular,  there is another RGD system for $G$,
 $(G,\{ U_\alpha \}_{\alpha \in -\Phi}, T)$, with the positive and negative root groups interchanged.  This induces a twin $BN$--pair structure on $G$ which guarantees
 a {\bf Birkhoff decompostion}
\beq
G = \coprod_{w\in W} B^+ w B^-
\label{genbirkhoff}
\nonumber
\eeq
\nod which when combined with the Bruhat decomposition for $P_I$ (\ref{bruhatdecomp}) and the fact that $B=U^+T$ easily leads to a {\bf generalized Birkhoff decomposition}

\beq
G = \coprod_{ wW_J \in W/W_J} U^+ w P_J^-
\label{genbirkhoff}
\eeq
\nod where $U^+ w P_J^- = \coprod_{v \in wW_J} U^+ w B^-$ (see \cite{foleythesis} for details).
 Using the standard RGD axioms, this symmetry between  positive and negative   is somewhat subtle.
Known proofs employ covering space theory \cites{Buildings,CR}.

The loc. cit. covering space arguments  imply for $n \in w \in W$
\beq
U^\pm \cap nU^\mp n^{-1}= U^\pm \cap wB^\mp w^{-1}= U^\pm_{w}
\label{uwstabs1}
\eeq
\nod where $U^+_{w}:= U_{w}$ and $U^-_{w} =\langle U_{-\alpha_{i_1}},  U_{-s_{i_1}\alpha_{i_2}}, \ldots  U_{-s_{i_1} \ldots s_{i_{k-1}}\alpha_{i_k}} \rangle$. More generally (see Lemma \ref{stablem}), for all $J \in \cS$ (\ref{Sposet})
\beq
 U^\pm \cap wP^\mp_J w^{-1}= U^\pm_{w_J}
\label{uwstabsgen}
\eeq
\nod where $w_J$ is the longest word in $wW_J$.

The groups $U_w$ also provide  improved Bruhat decompositions because
\beq
B^{\pm}wB^{\pm} = U_w^{\pm}wB^{\pm}
\nonumber
\eeq
\nod where expression on the left hand side factors \emph{uniquely} for a fixed choice of $w \in nT \in W$, cf. \cite[Lemma 8.52]{Buildings}, i.e.
\beq
G = \coprod_{w\in W} U_w^{+}wB^{+} = \coprod_{w\in W} U_w^{-}wB^{-}
\label{rgdbruhatdecomp}
\eeq
\nod so that all $g \in G$ factor uniquely as $g=uwb$ where $u\in U^\pm_w$, $b \in B^\pm$ and $w$ is an element in any fixed choice of coset representatives. Parabolic subgroups inherit similar expressions.

As previously mentioned, all groups with RGD systems have the structure of a twin $BN$--pair. This structure can be used to define
Levi component subgroups
\beq
G_I = P_I^- \cap P_I^+.
\label{levidef}
\eeq
The $G_I$ inherit a root group data structure $(G_I,\{ U_\alpha \}_{\alpha \in \Phi_I}, T)$ where $\Phi_I:=\{\alpha \in \Phi | U_\alpha \cap G_I \neq \{ e \} \}$.

\nod  When $I$ has finite type \cite{CR}, this leads to the semi-direct product decomposition
\beq
P_I^\pm \cong G_I \ltimes U_I^\pm
\label{levidecomp}
\eeq
 \nod known as the Levi decomposition with
 \beq
 G_I=\langle T, \hat{U}_I^{+}, \hat{U}_I^{+} \rangle
 \nonumber
 \eeq \nod for $\hat{U}_I^\pm = U^\pm \cap G_I$ called Levi component subgroups.
 Most of our principal applications will only require (\ref{levidecomp}) for finite type $I$.
For Kac-Moody groups,  (\ref{levidecomp}) holds for all $I$ and the submatrix $A_I:=(a_{ij})_{i,j \in I}$ of $A$ determines a Weyl group  $W_J$ group for $G_I$ via (\ref{kmweyl}).
 Moreover, $G_I$ is a reductive group exactly when $|W_I|< \ff$ \cite{kumar}.
See \cite{CR} for further discussion of when this decomposition is known to exist for general RGD systems.

We close the subsection with a simple but important lemma used in the proof of Theorem \ref{hocoUnewthmslick}.

\begin{lem}
For $P_J^-$ of finite type
\beq
U_{w P_J^- } = U_{w_J}
\eeq
\nod where $U_{w P_J^- }$ is the stabilizer of $\{ w P_J^- \}$ under the left multiplication action of $U^+$  on $G/P_J^-$, $w_J$ is the longest word in $wW_J$, 
and $U_w$ is defined in (\ref{uwdef}).
\label{stablem}
\end{lem}

\begin{proof}
 Note that $u w P_J^- = w P_J^-$ and $u \in U^+$ if and only if $u \in w P_J^-w^{-1} \cap U^+$.
Recall (\ref{rgdbruhatdecomp}) the  (improved) Bruhat decomposition for  $P_J^-$
\beq
P_J^- = \coprod_{v\in W_J} U_v^- v B^- .
\label{bruhatdecompJ}
\eeq
We compute
\beq
w P_J^-w^{-1} = w_J P_J^-{w_J}^{-1}  &=& \bigcup_{v\in W_J} w_J U_v^- v B^- {w_J}^{-1}  \nonumber \\ &=& \bigcup_{v\in W_J} (w_J  U_v^- {w_J}^{-1})  w_J v B^- {w_J}^{-1}.
\label{bruhatdecompeq1}
\eeq
\nod As $(w_J  U_v^- {w_J}^{-1}) \sset U^+$ for all $v \in W_J$ we have
\beq
w_J U_v^- v B^- {w_J}^{-1} \cap U^+ \cong   w_J v B^- {w_J}^{-1} \cap U^+.
\label{bruhatdecompeq2}
\nonumber
\eeq
Each $w_J v B^- {w_J}^{-1} \cap U^+$ isomorphic as a set to
\beq
w_J v B^-\cap U^+ {w}_J ~~ \sset ~~~ U^+ w_J v B^-\cap U^+ {w}_J B^-
\label{isobrik}
\eeq
\nod via right multiplication by $ {w_J}$.  Together (\ref{bruhatdecompJ}) and (\ref{isobrik}) imply
\beq
w_J v B^- {w_J}^{-1} \cap U^+ \neq \emptyset \Rightarrow v=e .
\label{bruhatdecompeq3}
\eeq
Combining (\ref{bruhatdecompeq1}--\ref{bruhatdecompeq3}) we have
\beq
U_{w P_J^- }= w P_J^-w^{-1} \cap U^+ &=& \bigcup_{v\in W_J} (w_J U_v^- v B^- {w_J}^{-1}\cap U^+) \nonumber \\&=&  w_J  U_e^- e B^- {w_J}^{-1} \cap U^+ \nonumber \\&=&w_J B^- {w_J}^{-1} \cap U^+= U_{w_J B^- } .\nonumber
\label{bruhatdecompeq4}
\eeq
Now, $(w_J B^- {w_J}^{-1} \cap U^+)= U_{w_J B^- }$  is known to be $U_{w_J}$ (\ref{uwstabs1}) via covering arguments that appear independently in \cite{Buildings} and \cite{CR}.
\end{proof}

\subsection{Pulling back homotopy colimits}
\label{sec:pbmain}

If $F:{\bf J} \rightarrow \cat{I}$ is a fixed functor, then we say $F$  {\bf pulls back homotopy colimits} if {\em for any} diagram of spaces $D:\cat{I} \rightarrow \cat{Spaces}$ the natural map
\beq
\xymatrix{
\hcl_{{\bf J}} DF \ar[rr]^-{\hcl_F} & & \hcl_{\cat{I}} D
}
\nonumber
\label{pbhoco}
\eeq
\nod is a weak homotopy  equivalence.
We will use Theorem \ref{coverholim} (see Section \ref{sec:subdia}) to pullback homotopy colimits over appropriate functors by assembling pullbacks over subcategories.
For  such a functor and any $i \in \cat{I}$ define the category $i \downarrow F$ as having objects
\beq
\Objects(i \downarrow F) = \{ (i \rightarrow i', j') | F(j')=i', i \rightarrow i' \in \Hom_\cat{I}\}
\nonumber
\label{underob}
\eeq
and morphisms
\beq &{}&
 \Hom_{i \downarrow F}((i \stackrel{g}{\rightarrow} i_1', j_1'), (i \stackrel{f}{\rightarrow} i_2', j_2)) \nonumber \\&=&
\{
( i_1' \stackrel{F(h)}{\rightarrow} i_2', j_1' \stackrel{h}{\rightarrow} j_2' |  F(h) g=f \in \Hom_\cat{I}, h \in \Hom_{\bf J})
\},
\label{undermor}
\nonumber
\eeq
\nod i.e.  morphisms are pairs of vertical arrow that fit into the following diagram
\beq
\xymatrix{
& i_1'= F(j_1') \ar[dd]_{F(h)} & &j_1' \ar[dd]^{h} \ar@{|->}[ll]\\
i \ar[dr]^{f} \ar[ur]^{g}& &     & \ar@{|=>}[ll]_F  \\
& i_2' = F(j_2')& &j_2' \ar@{|->}[ll]
}
\nonumber
\label{undermor2}
\eeq
\nod
  When $F$ is the identity on $\cat{I}$, we use the notation $i \downarrow \cat{I}$  and this definition reduces to that of an under category.

\begin{theoremoutside}[Pullback Criterion \cites{dk, hv}]
Let
$D:\cat{I} \rightarrow \cat{Spaces}$  be a diagram of spaces and $F:{\bf J} \rightarrow \cat{I}$ be a fixed functor, then the
following are equivalent:
\begin{itemize}
\item for all $D$ the canonical $\hcl_{{\bf J}} DF \stackrel{\thicksim}{\longrightarrow} \hcl_{\cat{I}} D$  is a weak equivalence,

\item for all $i \in \cat{I}$ the geometric realization of $i \downarrow F$ is weakly contractible.

\end{itemize}
\label{termin}
\end{theoremoutside}

In the applications of this paper,  ${\bf J}$ and $\cat{I}$ will be posets so that
$i \downarrow F= F^{-1}(i \downarrow \cat{I})$.

\subsection{Homotopy colimits over subdiagrams}
\label{sec:subdia}

At various times in this paper, we will wish to represent homotopy colimit presentations of a space in terms of homotopy colimits over subdiagrams of the main diagram.  This subsection will present tools to do this.

 An explicit model for the homotopy colimit of a small diagram $D: \cat{C} \rightarrow {\bf Top}$ of (topological) spaces $\hcl_\cat{C} D$ may be given as the union of
\beq
X_n=\frac{\coprod_{c \rightarrow c_1 \ldots \rightarrow c_{k\le n} \in \cat{C}} F(c) \times \Delta^k }{ \sim }
\label{hocolimit}
\eeq
\nod where $\Delta^k$ is the $k$--simplex and the relations correspond to the composition in $\cat{C}$ \cite{BK}.
 We call this the {\bf standard model} of a homotopy colimit.  Define the {\bf geometric realization} or classifying space of a category as
$
|\cat{C}|\simeq \hcl_\cat{C} \{\ast\}
 \label{nerve}
 $
 \nod for $\{\ast\}$ the one point space.

 We will also define a cover of a category $\cat{C}$ by subcategories $\{\cat{C}_i\}_{i \in I}$. 
 \begin{defn}
A cover of a small category $\cat{C}$ is a collection of subcategories $\{\cat{C}_i\}_{i \in I}$ such that, taking standard models (\ref{hocolimit}),
$\{|\cat{C}_i|\}_{i \in I}$ covers $|\cat{C}|$  (\ref{nerve}).
\label{catcover}
\end{defn}
\nod
We may also choose to enrich a cover by giving $I$ the structure of a poset such that $i\mapsto \cat{C}_i$ is a (commutative) diagram of inclusions of small categories.  Such a cover is given as a functor $\mathcal{U}: \cat{I} \rightarrow {\bf SubCat(C)}$ from a poset into the category subcategories of $\cat{C}$.  If there are no repetitions of $\cat{C}_i$'s, then $I$ may be canonically given
the poset structure induced by inclusion of subcategories.

 We will be interested in the taking homotopy colimits over subdiagrams
 \beq
  D_i= D|_{\cat{C}_i} : \cat{C}_i \longrightarrow {\bf Top} ~~~~
  \nonumber
  \eeq
  \nod and one of our main computational tools is provided by the following theorem.

  \begin{thm}
Let $D:  \cat{C} \rightarrow {\bf Top}$ be a diagram of topological spaces 
and  $\mathcal{U}: \cat{I} \rightarrow {\bf SubCat(C)}$ be a cover $\{ \cat{C}_i \}_{i\in \cat{I}}$ of $\cat{C}$ such that
\begin{itemize}
\item $\cat{I}$ has all greatest lower bounds and
\item $\mathcal{U}(\inf_A\{ \cat{C}_a \})= \cap_A \mathcal{U}(\cat{C}_a )$
\end{itemize}
\nod for all index sets $A$.  Then, $\hcl_{\cat{I}}(\hcl_{\cat{C}_i} D_i)$ and $\hcl_{\cat{C}} D$ are canonically, weakly equivalent.
\label{coverholim}
\end{thm}

By Thomason's Theorem \cite[Theorem 1.2]{thomason}, Theorem \ref{coverholim} may be reduced to the following.

  \begin{prop}
Let $D:  \cat{C} \rightarrow \cat{Spaces}$ be a diagram of spaces (e.g., topological spaces, $CW$--complexes, or simplicial sets)
and  $\mathcal{U}: \cat{I} \rightarrow {\bf SubCat(C)}$ be a cover $\{ \cat{C}_i \}_{i\in \cat{I}}$ of $\cat{C}$ such that
\begin{itemize}
\item $\cat{I}$ has all greatest lower bounds and
\item $\mathcal{U}(\inf_A\{ \cat{C}_a \})= \cap_A \mathcal{U}(\cat{C}_a )$
\end{itemize}
\nod for all index sets $A$.  Then, the canonical projection of the Groth\-en\-dieck construction (\ref{gr1}) 
of $\mathcal{U}$ onto  $\cat{C}$ pulls back homotopy colimits.
\label{coverholimprop}
\end{prop}

\begin{proof}
Recall that the Grothendieck construction 
of $\mathcal{U}$, which we will denote as $\cat{I}\ltimes \mathcal{U}$, is given by
\beq
\Objects(\cat{I} \ltimes \mathcal{U}) &=& \{ (i,c)| i \in \Objects(\cat{I}), ~~~ c \in \Objects(\cat{C}_i)\} , \\
 \Hom_{\cat{I} \ltimes \mathcal{U}}( (i,c),  (i' , c') ) &=& \{ (f,g)| f \in \Hom_{\cat{I}}(i, i'), ~~~ g \in \Hom_{C_{i'}}(c, c')\} \nonumber
 \label{gr1}
 \label{gr2}
\eeq
\nod with composition given by pushing forward along $\mathcal{U}$, i.e. $(f_2,g_2) \circ (f_1,g_1)= (f_2 f_1, g_2 \mathcal{U}(g_1) )$.
The projection $P: \cat{I} \ltimes \mathcal{U} \rightarrow \cat{C}$ maps $(i,c)$ to $c$.
By Theorem \ref{termin}, it is sufficient to  show that $|c \downarrow P|\simeq \pt $.
Let  $i(c)$ be the greatest lower bound for all categories in the cover containing $c$,
then $c \downarrow P$  has initial object $(c \stackrel{id}{\rightarrow}c, (i(c), c))$. 
\end{proof}

Dugger and Isaksen \cite{cech} present a model, up to weak equivalence, of a general topological space $X$ in terms of an arbitrary open cover ${\bf U}=\{ U_i\}_{i\in \cat{I}}$ and {\em finite} intersections of these $U_i$.
This allows a version of Theorem \ref{coverholim} to verified depending only on finite greatest lower bounds, but this version will not be needed here and is less natural in the sense that it relies on the standard weak equivalences generating
the model category structure.

Now, Theorem \ref{coverholim} allows the second criterion of Theorem \ref{termin} to be verified locally.
\begin{prop}
Let $F:\cat{J} \rightarrow \cat{I}$ be a map of posets such that $\cat{I}$ has all greatest lower bounds. If $F$ restricted to the pullback of all closed intervals in $\cat{I}$ pull backs homotopy colimits, then $F$
pull backs homotopy colimits.
\label{localprop}
\end{prop}
\begin{proof}
 A closed interval ${\bf [ }i_1 ,i_2{\bf ] }$ is defined to be the full subcategory of $\cat{I}$ with object set $\{ i| i_1 \rightarrow i \rightarrow i_2 \in \cat{I} \}$.  The category $i \downarrow \cat{I}$ is covered by $\{{\bf [ }i ,i'{\bf ] }\}_{ \exists i \rightarrow i'}$ and for any  set of objects $A$ the intersection
$\cap_{i'\in A} {\bf [ }i ,i'{\bf ] }= {\bf [ }i ,\inf_{\cat{I}} A{\bf ] } $.  Since covers and intersections pullback over functors, $\{F^{-1}{\bf [ }i ,i'{\bf ] }\}_{ \exists i \rightarrow i'}$ covers $\cat{J}$ and is closed under intersection.
  By hypothesis, $|F^{-1}{\bf [ }i ,i'{\bf ] }|\simeq \{\ast \}$. Theorem \ref{coverholim} implies  $|i \downarrow F= F^{-1}(i \downarrow \cat{I})|$ and $|i \downarrow \cat{I}|$
  have the same weak homotopy type. The result follows as $i \downarrow \cat{I}$ has an initial object, namely the identity on $i$.
\end{proof}

\subsection{Transport categories}
\label{sec:trans}
Transport categories are a specialization of the Groth\-en\-dieck construction \cite{thomason} for a small diagram of small categories.
\begin{defn}
Given a (small) diagram of sets $X: \cat{C} \rightarrow {\bf Sets}$, the transport category of $X$, denoted $\cat{Tr}(X)$ has
\beq
\Objects (\cat{Tr}(X)) &=& \{ (c,x_c) | c \in \cat{C},  x_c \in X(c) \}  \nonumber \\
\Hom_{\cat{Tr}(X)}((c,x_c), (c',y_{c'})) &=&  \{ f \in \Hom_{\cat{C}}(c,c')| X(f)(x_c)= y_{c'} \}
\nonumber
\label{transdefeq}
\eeq
\nod with composition induced by $\cat{C}$.
\label{transdef}
\end{defn}
For any Coxeter group $W$ with $\cS$ the set of subsets of generators that generate finite groups (\ref{Sposet}), our fundamental example of a transport category is $\cat{Tr}(X)$ for $X: \cS \rightarrow \cat{Sets}$ defined via $I \mapsto W/W_I$.
We call this category $\cat{W}_\cS$. Explicitly, $\cat{W}_\cS$ is a poset whose elements are finite type cosets and there exists a (unique) morphism from $wW_I$ to $vW_J$ precisely when $wW_I \sset vW_J$.
Its geometric realization $|\cat{W}_\cS|$ is the Davis complex \cite{Davis}.

We also wish to record a proposition  which is a specialization and enrichment of Thomason's  Theorem \cite[Thereom 1.2]{thomason}.  We indicate an explicit proof.
\begin{prop}
For any diagram of spaces over a transport category $D:\cat{Tr}(X)\rightarrow \cat{Spaces}$, there is a diagram of spaces $D':\cat{C}\rightarrow \cat{Spaces}$ over the underlying category such that
$\hcl_{\cat{Tr}(X)} D$ and $\hcl_\cat{C} D'$ are canonically weakly equivalent.
\label{transprop}
\end{prop}
\begin{proof} Let $D':\cat{C} \rightarrow \cat{Spaces}$ be defined by
$D'(c):= \hcl_{X(c)} D (c, x_c )$.  We have a natural weak equivalence
\beq
\hcl_{\cat{Tr}(X)} D' \stackrel{\sim}{\longrightarrow} \hcl_{\cat{C}} (\hcl_{X(c)} D|_{X(c)} )=\hcl_{\cat{C}} D'
\nonumber
\label{grotheq}
\eeq
\nod which can be observed by inspecting the standard models.  More explicitly, if one takes standard models the map (\ref{grotheq}) is specified
by the universal property of the homotopy colimit and the maps
\beq
D(c,x_c) \stackrel{D}{\longrightarrow} D(c',y_{c'}) \longrightarrow \coprod_{z_{c'} \in X(c')} D (c', z_{c'})
\nonumber
\eeq
\nod where the second map is the obvious inclusion.
\end{proof}

Note that Proposition \ref{transprop} gives a canonical equivalence between $\hcl_\cat{C} X$ and $|\cat{Tr}(X)|$.

\subsection{Coset geometry and colimits of groups}
\label{sec:cosetgeo}

Here we will see how homotopy colimit calculations in terms of subdiagrams apply to honest colimits. 

\begin{theoremoutside}[Seifert-van Kampen \cite{fun}]
Let $D: \cat{I} \rightarrow \cat{Spaces}_*$ be a diagram of pointed connected spaces 
such that $\cat{I}$ has an initial object.  Then, there is a natural isomorphism $\pi_1\hcl_\cat{I} D \cong \cl_\cat{I} \pi_1 D$.
\label{svkthm}
\end{theoremoutside}
  When $D: \cat{I} \rightarrow {\bf Groups}$ is  diagram of inclusions of subgroups of $H$, we  refer to the standard model of the homotopy fibre of  $\hcl_\cat{I} BD \rightarrow BH$ given by $\hcl_\cat{I} H/D(i)$ as the {\bf coset geometry} of the cocone. We call the transport category (see \ref{sec:trans}) for the functor $i \mapsto H/D(i)$ the
  {\bf poset form of the coset geometry} since its geometric realization is canonically equivalent to the coset geometry.
  This represents a mild generalization of a notion of coset geometry employed by others, such as Tits \cite{Titscover} and Caprace and Remy \cite{CR}.  For instance,  our notion permits diagrams whose image in the lattice
   of subgroups is not full and our poset is directed in the oppositive direction.

We can always add the trivial group to any diagram of inclusions of groups in an initial position  without affecting the colimit of that diagram. Thus, verifying any presentation of a group as a colimit of subgroups
reduces to homotopy knowledge of the coset geometry.  An example of this technique gives the following additional corollary.

\begin{cor}
Let $D: \cat{I} \rightarrow \cat{Groups}$ be a diagram of the subgroups of $H$ with each map an inclusion of subgroups such that $|\cat{I}|$ is simply connected.  The canonical $\cl_\cat{I} D \stackrel{\thicksim}{\longrightarrow} H$ is an isomorphism
if and only if its coset geometry is simply connected.
\label{simcon}
\end{cor}
\begin{proof} We may artificially add an initial object, $\bullet$, to $\cat{I}$ and call this new category $\cat{I} \cup \bullet$. We can extend $D$ to $\cat{I} \cup \bullet$ by sending $\bullet$ to the trivial group.
Up to homotopy, the space $Y:=\hcl_{\cat{I} \cup \bullet} BD$ is obtained from $X:=\hcl_{\cat{I}} BD$ by coning off the subspace $Z:=\hcl_{\cat{I}} \{\ast\}$, the homotopy colimit base points.
Using Theorem \ref{svkthm}, the fundamental group $\pi_1 (\hcl_{\cat{I} \cup \bullet} BD) \cong  \cl_\cat{I} D$. Since $\pi_1(Z)=0$, $\cl_\cat{I}$ agrees with $ \pi_1 (\hcl_{\cat{I}} BD)$.
 By the long exact sequence of a fibration, the coset geometry of $D$ is simply connected if and only if the canonical $\cl_\cat{I} D \rightarrow H$ is an isomorphism.
\end{proof}

Compare \cite[Lemma 1.3.1]{DC} and \cite[\S3]{Hal}. For example, as observed in \cite{CR}, if $W$ is a Coxeter group with $\cS_2$ the set of subsets of generators of cardinality at most $2$ that generate finite groups,
then the colimit presentation $\cl_{\cS_2} W_I \cong W$, which is immediate from the presentation of $W$ (\ref{cox}), implies that the associated coset geometry is simply connected.

\begin{rem}
When the conditions for Corollary \ref{simcon} are satisfied,  a homotopy decomposition of a classifying space in terms of the classifying spaces of subgroups induces a colimit presentation.  Though not stated explicitly,  the homotopy decompositions of
Theorem \ref{hocoUnewthmslick}, Lemma \ref{hocoUnewthm}, and  Corollary \ref{Bruhatorderdecomphat} as well as the known homotopy decompositions (\ref{hocoBW}--\ref{hocoBG}) all induce colimit presentations.
Of course, a homotopy decomposition a strictly stronger than a colimit presentation. The former is characterized by having a contractible coset geometry whereas latter only
requires a  simply connected coset geometry.
For instance, $\cl_{\cS_2} W_I \cong W$ is induced by a homotopy decomposition of $BW$ only if $\cS_2 =\cS$.
\label{rem:cosetgeo}
\end{rem}

\subsection{Known homotopy decompositions}
\label{sec:known}

For any Coxeter group $W$ with $\cS$ (\ref{Sposet}) the poset of subsets of generators that generate finite groups
\beq
\hcl_{I\in \cS} BW_I \stackrel{\thicksim}{\longrightarrow} BW
\label{hocoBW}
\eeq
\nod is a homotopy equivalence including, trivially,  the cases where $W$ is finite. For non-trivial decompositions, the associated coset geometry is Davis's version of the Coxeter complex, $|\cat{W}_\cat{S}|$, which is known to be contractible  \cites{Davis, moussongthesis}. Theorem \ref{combin} below provides a combinatorial proof.
  Likewise, for any discrete $BN$--pair the canonical map induced by inclusion of subgroups
\beq
\hcl_{I\in \cS}\BPI \stackrel{\thicksim}{\longrightarrow} BG
\label{hocoBG}
\eeq
\nod is a homotopy equivalence including, trivially, the cases where $W$ is finite. Here non-trivial decompositions have Davis's version of the Tits building as coset geometries. 
  For completeness, we also note that (\ref{hocoBW}) and (\ref{hocoBG}) can be deduced from the contractibility of the usual Coxeter complex (resp. Tits building) inductively because this provides a method to verify (\ref{hocoBG}) for complex topological Kac-Moody groups. More specifically, for any topological $G$, a $BN$--pair with closed parabolic subgroups and  infinite Weyl group $W$ the canonical
map $\hcl_{I\in \cR} \BPI {\longrightarrow} BG$ for $\cR$ the set of all proper subsets of the generating set of $W$ has contractible coset geometry \cite{contract}. Thus, this map is a weak homotopy equivalence.
Because parabolic subgroups inherit the structure of a $BN$--pair and inductively (on $|I|$) have homotopy decompositions (\ref{hocoBG}) in terms of finite type parabolics,  Theorem \ref{coverholim} may be applied to obtain
(\ref{hocoBG}) for $BG$ whenever all parabolic subgroups are closed. This proof of (\ref{hocoBG}) for topological $BN$--pairs are implicit in \cites{classKM, nituthesis, nitutkm, domkt}.
Note that  (\ref{hocoBW}--\ref{hocoBG}) are sharp in the sense that any proper subposet of $\cS$ will not suffice to give a homotopy decomposition.

Recall that each parabolic subgroup of a Kac-Moody group has a semi-direct product decomposition $P_I \cong G_I \ltimes U_I$ (\ref{levidecomp}) with these $G_I$ called Levi component subgroups.
  For $K$ a topological Kac-Moody group (with its topology induced by the complex numbers) the subgroups $U_I$ for $I$ finite type are known to be contractible \cites{kumar}.
This fact with  (\ref{levidecomp}) and  (\ref{hocoBG}) implies that the inclusion of the Levi component subgroups
induces a  homotopy decomposition for topological Kac-Moody group classifying spaces
\beq
\hcl_{I\in \cS}\BGI  \stackrel{\thicksim}{\longrightarrow} \hcl_{I\in \cS}\BPI \stackrel{\thicksim}{\longrightarrow} BK
\label{hocoBGcomplex}
\eeq
\nod where $G_I$ are reductive, Lie, Levi component subgroups, cf. \cites{nituthesis, nitutkm, domkt}. This decomposition has been
one of the fundamental tools for the study $BK$ using homotopy theory \cites{pcompact, rank2mv, BKtoBK, classKM, nitutkm}.

Though discrete $U_I(R)$ over a commutative ring $R$ are in general quite far from being contractible (see \ref{sec:calcp}), they are expressible in terms of honestly unipotent subgroups, i.e. each subgroup can be embedded into
the subgroup of $GL_n(R)$ of upper triangular matrices with unit diagonal for some $n$. The next section will make this precise.

\section{Homotopy Decompostions}
\label{sec:homodecomp}

The positive unipotent subgroup of a discrete Kac-Moody group $U^+=U_\emptyset$ is expressible \cites{CR, Buildings}
as a colimit of finite dimensional unipotent algebraic subgroups $U_v$ (\ref{uwdef})
\beq
\cl_\cat{W} U_{w} \cong U^+
\eeq
as conjectured by Kac and Petersen \cite{relations}. This colimit presentation is induced (see Remark \ref{rem:cosetgeo}) by an integral homotopy decomposition
and in fact all $BU_I(R)$ are expressible in terms of unipotent subgroup classifying spaces.
The proof relies only on the combinatorics of the RGD system (see \ref{sec:rgd}) associated to a Kac-Moody 
and applies to many variants of Kac-Moody groups (cf. \cites{CR, Buildings} for discussion) as well other groups with  RGD systems.

\begin{thm}
Let $U^+=U_\emptyset$ be the positive unipotent subgroup of a discrete Kac-Moody group whose Weyl group is viewed as a poset $\cat{W}$ with respect to the weak Bruhat order (\ref{weakorder}).  
Then,  the canonical map induced by inclusion of subgroups
\beq
\hcl_\cat{W} BU_{w} \stackrel{\thicksim}{\longrightarrow} BU^+
\label{hocoUnewslick}
\eeq
\nod induces a homotopy equivalence where $U_v$ (\ref{uwdef}) are finite dimensional unipotent algebraic subgroups.
More generally, $
\hcl_{\hat{U}^+_I \cdot \cat{W}} BU_{(\hat{u}, w)} \stackrel{\thicksim}{\longrightarrow} BU_I^+
\label{hocoUInewslick}
$ where $\hat{U}^+_I \cdot \cat{W}$ is a poset with a contractible geometric realization and these $U_{(\hat{u}, w)}$ are unipotent.
\label{hocoUnewthmslick}
\end{thm}

Note that  the induced colimit presentations
corresponding to $I \neq \emptyset$ are new. Theorem \ref{hocoUnewthmslick} will be shown by using alternative homotopy decompositions indexed over posets with finite chains given in Lemma \ref{hocoUnewthm} below.

To accomplish this reduction, we relate diagrams over the poset $\cat{W}$ with certain diagrams over the poset $\cat{W}_\cS$ (see \ref{sec:trans}) which underlies
the Davis complex \cite{Davis} of a Coxeter group. 
\begin{thm}
For any Coxeter group $W$, the longest element functor $L:\cat{W}_\cS \rightarrow \cat{W}$, $vW_I \mapsto v_J$ (\ref{Sposet}),
 pullbacks homotopy colimits so that the natural map \\$\hcl_{\cat{W}_\cS} DL \longrightarrow \hcl_{\cat{W}} D$
is a weak homotopy  equivalence {\em for any} diagram of spaces $D:\cat{W} \rightarrow \cat{Spaces}$.
\label{combin}
\end{thm}
This immediately implies that
 the Davis complex is contractible  since the poset $\cat{W}$ has the identity of $W$ as its initial object.
 When $W\cong (\ZZ/2\ZZ)^{\ast n}$,
 Theorem \ref{combin} lets us replace a diagram over an infinite depth tree ($\cat{W}$) with a canonical diagram over depth one tree ($\cat{W}_\cS$) obtained via barycentric subdivision. Theorem \ref{combin} allows this
 procedure to be extended to posets indexed by more general Coxeter groups.

\subsection{The proof of Theorem \ref{combin}}
\label{sec:pb}

Our proof of Theorem \ref{combin} will inductively  pullback homotopy colimits over closed intervals and apply Proposition \ref{localprop}. For instance, the following corollary---which will be used in \ref{sec:calcp}---is immediate from our proof by Proposition \ref{localprop}.
\begin{cor}
With the definitions of Theorem \ref{combin}, let $\cat{X} \sset \cat{W}$ be a full subposet covered by a collection of  intervals $[v,w]$,
 then the longest element map $L|_{L^{-1}\cat{X}}$ pulls back homotopy colimits.
\label{combinlocal}
\end{cor}
We may also extend Theorem \ref{combin} to situations where $\cat{W}$ is the fundamental domain of a group action on a poset.
 For our purposes here, we will take $W$ to be Weyl group of a group with RGD system. We define $V \cdot \cat{W}$ for any $V\le U^+$ to be the $V$--orbit under left multiplication of ${\bf  W}$ realized as $\{ U_w \}_{w\in W}$ within the poset of cosets of subgroups of $U^+$.  We also define $V \cdot\cat{W}_\cS$
so that
$L$ extends to a $V$--equivariant functor $V\cdot L:V\cdot \cat{W}_\cS \rightarrow V\cdot \cat{W}$, i.e. $wW_I$ corresponds to $U_{w_I}$.
  \begin{cor}
Let $\cat{W}$ be realized as the poset of subgroups $\{ U_w \}_{w\in W}$ ordered by inclusion.   Then
 $V\cdot L:V \cdot  \cat{W}_\cS \rightarrow V \cdot  \cat{W}$ pulls back homotopy colimits for any $V\le U^+$.
\label{combinlocalact}
\end{cor}
\begin{proof}
 Observe  that $V \cdot  \cat{W}$ is the transport category for $X: \cat{W} \rightarrow \cat{Sets}$ defined by $w \mapsto V/(V \cap U_{w})$
and $V \cdot  \cat{W}_\cS$ is that transport category for $X_\cS: \cat{W}_\cS \rightarrow \cat{Sets}$ defined by $wW_I \mapsto V/(V \cap U_{w_I})$.
As in the proof of Proposition \ref{transprop}, any fixed functor $V \cdot  D:V \cdot  \cat{W} \rightarrow \cat{Spaces}$ is associated to $D:\cat{W} \rightarrow \cat{Spaces}$ defined by
$
D(w):= \coprod_{X(w)}  D (v (V \cap U_{w} ))
$.  If $V \cdot  D$ is pulled back along $V \cdot  L$, then we obtain $(V \cdot  D)(V \cdot  L):V \cdot  \cat{W}_\cS \rightarrow \cat{Spaces}$ which is associated to
$DL:\cat{W}_\cS \rightarrow \cat{Spaces}$ by the same procedure. Thus, we have a diagram
\beq
\xymatrix{
\hcl_{V \cdot  \cat{W}_\cS} (V \cdot  D) (V \cdot  L) \ar[d]^{V \cdot  L} \ar[r]^-{\sim} &  \hcl_{\cat{W}_\cS} \ar[d]^{L} DL \\
\hcl_{V \cdot  \cat{W}} V \cdot  D \ar[r]^-{\sim} & \hcl_\cat{W} D
\label{transportcat}}
\nonumber
\eeq
\nod that commutes up to homotopy and the horizontal maps are weak equivalences by Proposition \ref{transprop}.
Now, by Theorem \ref{combin}, the right vertical map is a weak equivalence and $V \cdot  L$ pulls back homotopy colimits.
\end{proof}

For instance, with a bit of reflection, we see that the functor $U^+\cdot L:U^+\cdot\cat{W}_\cS \rightarrow U^+\cdot\cat{W}$ directly relates the poset forms of the coset geometries associated to the homotopy decompositions of Lemma \ref{hocoUnewthm} (appearing below) and Theorem \ref{hocoUnewthmslick} in the $I=\emptyset$
case.
We also note a useful observation.
  \begin{prop}
Define ${w(I)}$ as the longest $v \in W_I$ such that $v \le w$ in $\cat{W}$.  The functor  $L_I: \cat{W} \rightarrow \cat{W}_I$ given by $w \mapsto w(I) $ pulls back homotopy colimits.  Moreover,
$L|_{L^{-1}[w,v]}$ pulls back homotopy colimits for all intervals and $V \cdot  L_I$ pulls back homotopy colimits for all $V\le U^+$ as in Corollary \ref{combinlocalact}.
\label{w(I)prop}
\end{prop}
\begin{proof}
 Note that $
w \downarrow L_I = L_I^{-1} (w(I) \downarrow \cat{W}_I)= w(I) \downarrow \cat{W}
$
and $w(I) \downarrow \cat{W}$ has initial object $w(I)$.  For closed intervals, $L_I^{-1} ([w,v])$ still has initial object $w(I)$. The functor $V \cdot  L_I$ is defined as in
Corollary \ref{combinlocalact} with respect to the base $\cat{W}_I$, i.e. $\cat{W}_I$ is realized
as the poset of subgroups $\{ U_w \}_{w\in W_I }$.
With this definition, $V \cdot  L_I$ pulls back homotopy colimits for all $V\le U^+$ as in the proof Corollary \ref{combinlocalact}.
\end{proof}

\begin{proof}[Proof of Theorem \ref{combin}]
 By Proposition \ref{localprop} it is sufficient to show that the geometric realization of the poset
 \beq
 \cat{X}^v_w := L^{-1} [v, w] .
 \label{xvw1}
 \nonumber
 \eeq
 \nod is contractible for all $v \le w \in \cat{W}$.
 Define $l[v,w]$ to be the maximal chain length in $[v,w]$. Let us proceed by induction on this length.  If
$v=w$, then   $wW_{I_w}$ is the (unique) terminal object of $\cat{X}^{v}_{v=w}$ and $|\cat{X}^{v}_{v=w}|\simeq \{ \ast \}$.

Fix $(v,w)$ with $v < w$.
We will cover $\cat{X}^v_w$ as a category and apply Theorem \ref{coverholim} to show that $|\cat{X}^v_w| \simeq \pt$, inductively.  Let us start to define the elements of this cover precisely.

For any $w \in \cat{W}$, there is a unique greatest $I_w\in \cS$ such that $w$ is longest in $wW_{I_w}$.  By Theorem \ref{wordproblem}, $I_w$ is precisely the set of all $i$ such that $s_i$ is the right most letter of some reduced word expression of $w$.  Define ${\bf Y}_w$ to be the full subcategory of the poset $\cat{W}_\cS$ with
\beq
\Objects ({\bf Y}_w) = \{ vW_J \in W_\cS |  vW_J \le wW_{I_w}\}.
\label{Yw}
 \nonumber
 \eeq
Now, $\cat{X}^v_w$ is a full subcategory of $\cat{W}_\cS$ and
\beq
\Objects (\cat{X}^v_w) \sset \bigcup_{v\le x \le w} \Objects ({\bf Y}_x)
\label{Xvw}
 \eeq
\nod where $v\le w$ refers to the weak Bruhat order on $W$ (\ref{weakorder}).  Because all chains in $\cat{W}_\cS$ are contained within some ${\bf Y}_x$,
$\cat{X}^v_w$ is covered by $ \{ {\bf Y}_x\cap \cat{X}^v_w\}_{{v\le x \le w}}$, as a category (see Definition \ref{catcover}).

 We now see that   $\cat{X}^v_w$ is covered as a category by $\cat{X}^v_{ws_i}$ for all $i \in I_w$ and  $\{ \cat{X}^v_w \cap {\bf Y}_w \}$
since
\beq
  \bigcup_{i \in I_w} \Objects (\cat{X}^v_{ws_i}) &=& \bigcup_{\stackrel{v\le x \le ws_i}{i \in I_w}} \Objects ({\bf Y}_x \cap \cat{X}^v_{ws_i}) \nonumber \\
  &=& \bigcup_{{v\le x < w}} \Objects ({\bf Y}_x \cap \cat{X}^v_{w}).
  \nonumber
\eeq
and $\cat{X}^v_w$ is covered by $ \{ {\bf Y}_x\cap \cat{X}^v_w\}_{{v\le x \le w}}$ by (\ref{Xvw}).

Remove all empty elements of this cover.  In particular, $\cat{X}^v_{ws_i}$ is non-empty only if $v \le ws_i$.
Define $y$ to be the shortest element of $wW_{I_w}$.
Any (non-empty) element of this cover contains the singleton coset $\{ z:=\sup \{ v, y \}\}$  which exists and is less than or equal to $w$ since $v < w$ and $y < w$.
  Thus, we can close this cover under intersection without introducing empty categories.

If $m\le |I_w|+1$ is the number of elements of this cover, we may define $\mathcal{U}: \Delta_{m+1} \rightarrow \cat{Cat}$ via $\{i_1, \dots i_k\} \mapsto \cat{U}_{i_1} \cap \ldots \cap \cat{U}_{i_k}$ for some enumeration of this cover. Here $\Delta_{m+1}$ is the category of inclusions of facets in the standard $m+1$--simplex.  By Theorem \ref{coverholim}, $\hcl |\mathcal{U}|$ is weakly homotopy equivalent to $|\cat{X}^v_w|$.  As $|\Delta_m| \simeq \{\ast\}$,
it is enough to show that each $|\cat{U}_{i_1} \cap \ldots \cap \cat{U}_{i_k}| \simeq \{\ast\}$ by induction on $l[v, w]$.  In fact, we will show that each element of this cover is isomorphic  to some $\cat{X}^{v_1}_{v_2}$ with $l[{v_1},{v_2}] < l[v,w]$ or has a terminal object.

{\bf Case 1:} $\bigcap_{i\in I \sset I_w} \cat{X}^v_{ws_i}$. Any non-empty intersection of $\cat{X}^v_{ws_i}$ for $i \in I_w$ is equal to some $\cat{X}^v_{x}$ with $x < w$   because
any intersection of intervals $[v,ws_i]$ will be  some interval $[v,x]$ and intersections pullback along functors.  Inductively, $|\cat{X}^v_{x}| \simeq \{\ast\}$.

{\bf Case 2:} $\bigcap_{i\in I \sset I_w} \cat{X}^v_{ws_i} \cap {\bf Y}_w$.
Any  intersection of  ${\bf Y}_w$ and at least one non-empty $\cat{X}^v_{ws_i}$ with $i \in I_w$ will be equal to $\cat{X}^v_x \cap {\bf Y}_w$ for some $x < w$.
  Multiplication by $y$ induces an isomorphism
of posets \beq \cat{X}^{e_{I_w}}_{e_{I_w}} \stackrel{\simeq}{\longrightarrow} {\bf Y}_w
 \label{iso}\eeq
\nod for $e_{I_w}$ the longest word in $W_{I_w}$.
Now, $\cat{X}^v_x \cap {\bf Y}_w = \cat{X}^z_x \cap {\bf Y}_w$ for $z=\sup \{ v, y \}$ and $\cat{X}^z_x \cap {\bf Y}_w$ is in bijection with $\cat{X}^{y^{-1}z}_{y^{-1}x}$ under (\ref{iso}).
Observe that $  l[y^{-1}z, y^{-1}x]=  l[z, x] \le l[v, x] < l[v, w]$ as $ v \le z \le x < w$.
By induction,
\beq
|\cat{X}^v_x \cap {\bf Y}_w = \cat{X}^z_x \cap {\bf Y}_w \cong \cat{X}^{y^{-1}z}_{y^{-1}x}|\simeq \{\ast\}.
\nonumber
\eeq

{\bf Case 3:} $\cat{X}^v_{w} \cap {\bf Y}_w$. In this case, there is terminal object, namely $wW_{I_w}$, and $|{ \bf X}^v_{w} \cap {\bf Y}_w| \simeq \pt$.
This completes the proof.
\end{proof}

\subsection{New homotopy decompositions for groups with RGD systems}
\label{sec:newdecomposition}

Theorem \ref{hocoUnewthmslick} will follow from Lemma \ref{hocoUnewthm}  and Lemma \ref{combocontract} (below) by pulling back appropriate homotopy colimits.

\begin{lem} Let $U^+=U_\emptyset$ be the positive unipotent subgroup of a group with RGD system and $\cat{W}_\cS$ be the poset underlying the Davis complex 
 (see \ref{sec:trans}).
 Then  the canonical map
\beq
\hcl_{\cat{W}_\cS} BU_{w_J} \stackrel{\thicksim}{\longrightarrow} BU^+
\label{hocoUnew}
\eeq
\nod induces a homotopy equivalence where $w_J$ is the longest word in $wW_J$ under the weak Bruhat order (\ref{weakorder}).
More generally, if the standard parabolic subgroup $P_I$ has a Levi decomposition then \beq
\hcl_{\hat{U}^+_I\cdot \cat{W}_\cS} BU_{(\hat{u}, w)} \stackrel{\thicksim}{\longrightarrow} BU_I^+
\nonumber
\label{hocoUInew}
\eeq
\nod where $\hat{U}^+_I\cdot \cat{W}_\cS$ is the poset defined in Corollary \ref{combinlocalact} and each $U_{(\hat{u}, w)}$ is isomorphic to a subgroup of some $U_v$.
\label{hocoUnewthm}
\end{lem}
\begin{proof}
 Let us first show the $I=\emptyset$ case.  We calculate
\beq
BU^+ &\simeq& EU^+ \times_{U^+} \{ \ast \} \simeq EU^+ \times_{U^+} \hcl_\cS G/P^-_J \nonumber \\
&\simeq&  \hcl_\cS EU^+ \times_{U^+} G/P^-_J \nonumber \\
&\simeq&  \hcl_\cS (\coprod_{w\in wW_J \in W/W_J} EU^+ \times_{U^+} U^+ w P^-_J) \nonumber \\
&\simeq&  \hcl_{ \cat{W}_\cS} EU^+ \times_{U^+} U^+ w P^-_J \nonumber \\
&\simeq&  \hcl_{\cat{W}_\cS} B( \Stab_{U^+} \{wP^-_J\} )
\label{quickway}
\eeq
\nod where the fourth equivalence requires the generalized Birkhoff decomposition (\ref{genbirkhoff}) and the fifth uses the isomorphism of posets $w P^-_J \mapsto w W_J$ and Proposition \ref{transprop} for
$X:\cS \rightarrow \cat{Sets}$ defined via $I \mapsto W/W_I$.   Alternatively, it is not overly difficult to check directly that the map $uwP_J^- \mapsto (wW_J, u\Stab_{U^+} \{wP^-_J\})$ is an isomorphism of the poset forms of the coset geometries associated to $\hcl_{\cat{W}_\cS} B(\Stab_{U^+} \{wP^-_J\})$ and $\hcl_{\cS} BP^-_J$, respectively.
Note that Lemma \ref{stablem} identifies the stabilizer $\Stab_{U^+} \{wP^-_J\}$ as $U_{w_J}$.  This completes the $I=\emptyset$ case.

Whenever $P_I$ has a Levi decomposition (\ref{levidecomp}), $U_I^+ \rtimes \hat{U}_I^+ \cong U^+ \subset P_I^+ \cong  U_I^+ \rtimes G_I$  so that
\beq
G = \coprod_{w\in W/W_J} U^+ w P_J^-= \coprod_{w\in W/W_J} U_I^+ \hat{U}_I^+ w P_J^-
\nonumber
\label{genbirkhoffI}
\eeq
\nod by the generalized Birkhoff decomposition (\ref{genbirkhoff}).  Thus,
\beq
BU_I^+ &\simeq&  EU^+ \times_{U_I^+} \hcl_\cS G/P^-_I \nonumber \\
&\simeq&  \hcl_{\hat{U}_I^+\cdot \cat{W}_\cS} B( \Stab_{U^+_I} \{ \hat{u} wP^-_I\} )
\nonumber
\label{quickwayact}
\eeq
\nod as in (\ref{quickway}).  Here the map $u\hat{u}wP_J^- \mapsto (wW_J, u\Stab_{U^+_I} \{\hat{u}wP^-_I\})$ is an isomorphism of the poset forms of the coset geometries associated to $\hcl_{\hat{U}_I^+ \cdot \cat{W}_\cS} B( \Stab_{U^+_I} \{ \hat{u} wP^-_I\} )$ and $\hcl_{\cS} BP^-_J$, respectively, for $u \in U^+_I$ and $\hat{u} \in \hat{U}_I^+$.

Let us characterize
$\Stab_{U_I^+} \{ \hat{u} w P_J^- \} = U_I^+ \cap \hat{u} w P_J^- w^{-1} \hat{u} ^{-1}$. We see
\beq
 \hat{u} ^{-1} U_I^+ \hat{u} \cap  w P_J^- w^{-1} &=& \hat{u} ^{-1} U_I^+ \hat{u} \cap U^+  \cap  w P_J^- w^{-1}  \nonumber \\
 &=& \hat{u} ^{-1} U_I^+ \hat{u}  \cap U_{w P_J^-}.
 \label{stab2}
\eeq
\nod Thus
$\Stab_{U_I^+} \{ \hat{u} w P_J^- \} =\hat{u} U_{w P_J^-} \hat{u} ^{-1} \cap  U_I^+$. As $J$ has finite type, each $
\hat{u} U_{w P_J^-} \hat{u} ^{-1}= \hat{u} U_{w_J}^- \hat{u} ^{-1} \cong U_{w_J}$ by Lemma \ref{stablem}.
  This completes the proof.
\end{proof}

Recall that $G_I$ (\ref{levidef}) has a root group data structure with Weyl group $W_I$.  The positive unipotent subgroup of this
root group data structure is $\hat{U}_I^+$. Applying Theorem \ref{combin} to (\ref{hocoUnew})
 gives the following.

\begin{cor}
Under the assumptions of Lemma \ref{hocoUnewthm}, the canonical map induced by inclusions of subgroups
\beq
\hcl_{w\in \cat{W}_I} BU_w \stackrel{\thicksim}{\longrightarrow} B\hat{U}_I^+
\nonumber
\label{hocoUwhat}
\eeq
\nod is a homotopy equivalence where $\cat{W}_I$ is a poset under the weak Bruhat order (\ref{weakorder}).
\label{weyldchat}
\label{Bruhatorderdecomphat}
\end{cor}

We are now ready to prove our final lemma for Theorem \ref{hocoUnewthmslick}.

\begin{lem}
The posets $\hat{U}^+_I\cdot \cat{W}_\cS$ and $\hat{U}^+_I\cdot\cat{W}$, defined in Corollary \ref{combinlocalact},  have contractible geometric realizations.
\label{combocontract}
\end{lem}
\begin{proof}
 Recall the definition of ${w(I)}$ as the longest $v \in W_I$ such that $v \le w$ in $\cat{W}$.
Note that
(\ref{Uwintersect}) implies $U_{w} \cap \hat{U}_I^+ =  \bigcap_{v \in W_I} {U}_v \cap U_{w} = U_{w(I)}$.  Thus, we have a commutative diagram of fibrations over $BU^+_I$ induced by inclusions of subgroups
\beq
\xymatrix{
|\hat{U}_I^+ \cdot \cat{W}_\cS| \ar[d]_{\hat{U}_I^+ \cdot L} \ar[r]&  \hcl_{\cat{W}_\cS}  B(U_{w_J} \cap \hat{U}_I^+) \ar[d]_L \ar@/^/[ddr] & \\
|\hat{U}_I^+ \cdot \cat{W}| \ar[d]_{\hat{U}_I^+ \cdot L_I} \ar[r]&   \hcl_{\cat{W}} B(U_{w} \cap \hat{U}_I^+) \ar[d]_{L_I}  \ar[dr] & \\
|\hat{U}_I^+ \cdot \cat{W}_I| \ar[r]&  \hcl_{\cat{W}_I} BU_{w(I)} \ar[r]  & B \hat{U}_I^+ \\
|\hat{U}_I^+ \cdot (\cat{W}_I)_{\cS \cap I}| \ar[r] \ar[u]^{\hat{U}_I^+ \cdot L}&  \hcl_{(\cat{W}_I)_{\cS \cap I}} BU_{w} \ar[ur] \ar[u]^L & \\
\label{3pullback}} \nonumber
\eeq
\nod with the vertical maps induced by the indicated functors of index categories as defined in \ref{sec:pb}.  In particular, the vertical maps are weak equivalences by Theorem \ref{combin}, Corollary \ref{combinlocalact} and Proposition \ref{w(I)prop}.  The bottom fibration is simply the homotopy decomposition (\ref{hocoUnew}) of Lemma \ref{hocoUnewthm} for the (not necessarily compact) Levi factor $G_I$ which carries a RGD structure
with positive unipotent subgroup $\hat{U}_I^+$. Thus, 
 we have $|\hat{U}_I^+ \cdot \cat{W}_\cS|\simeq |\hat{U}_I^+ \cdot\cat{W}| \simeq |\hat{U}_I^+ \cdot\cat{W}_I| \simeq \pt$ which completes
the proof.
\end{proof}

\begin{proof}[Proof of Theorem \ref{hocoUnewthmslick}]
Note that for discrete Kac-Moody groups the subgroups $U_w$ are unipotent  and
Levi decompositions always exist (see Section \ref{sec:rgd}).
Lemma \ref{hocoUnewthm} implies that all the statements of Theorem \ref{hocoUnewthmslick} follow from Theorem \ref{combin} and Corollary \ref{combinlocalact} except
the claim that $|\hat{U}^+_I\cdot\cat{W}|$ is contractible.  This final claim is shown in Lemma \ref{combocontract}.
\end{proof}

\section{Main Results}
\label{sec:mainresults}

We now can express the classifying space of the unipotent factor of any parabolic subgroup of a discrete Kac-Moody group $BU_I^+(R)$ in terms of a countable collection of classifying spaces of unipotent subgroups.
In particular, when all these unipotent subgroups have vanishing homology, our decompositions imply that the homology of $BU_I^+(R)$ must vanish. This provides sufficient input to express $BK(R)$
as in
Theorem \ref{introhocoBGqfinite}.  For simplicity, 
 we state our results in this section for Kac-Moody groups over fields.

\subsection{Vanishing and simplification away from $p$}
\label{sec:van}

Our main application of the results of \ref{sec:newdecomposition}   is Theorem \ref{vanishthm}.
\begin{proof}[Proof of Theorem \ref{vanishthm}]
  Fix fields $\FF$ and $\LL$ of different characteristics.
  First note that $BU_w(\FF)$ is a $\LL$--homology point since each $U_w(\FF)$ has a normal series of length $l(w)$, i.e. the reduced word length of $w$, with quotient groups isomorphic to $(\FF,+)$.
  Of course, $H_n(\FF, \LL)=0$ for all $n>0$, cf. \cite[Theorem 6.4, p. 123]{Brown}.
   For example, if $\FF=\Fpk$ each $U_w(\Fpk)$ is a $p$--group with $p^{kl(w)}$ elements; compare (\ref{multiply}).

   By the homology spectral sequence for homotopy colimits \cite{BK},  $H_\ast (BU^+(\FF), \LL)$ is
the homology of the poset $\cat{W}_\cS$ underlying the homotopy decomposition (\ref{hocoUnew}).   Since the Davis complex $|\cat{W}_\cS|$ is contractible, $H_n(BU^+(\FF), \LL)=0$ for all $n>0$, completing the proof in the case of $I=\emptyset$. Theorem \ref{hocoUnewthmslick} implies that same proof will work for all $I$
since all $U_{(\hat{u}, w)}(\FF)$ are unipotent and $|\hat{U}_I^+(\FF)\cdot \cat{W}_\cS|\simeq \pt$. 
\end{proof}

An  independent proof of the rank two case of Theorem \ref{vanishthm} appears in \cite{AR}.
Notice that the proof of Lemma \ref{combocontract} brings together the full combinatorial tool kit to show that $|\hat{U}_I^+(\FF)\cdot \cat{W}_\cS|$ has vanishing homology.
  Theorem \ref{KMbar} will follow
from an analog at $q$ of the decomposition of a topological Kac-Moody group  (\ref{hocoBGcomplex}) for Kac-Moody groups over a fields of characteristic $p$.
We now state a precise version of Theorem \ref{introhocoBGqfinite}.

\begin{thm}
Let $K(\FF)$  be a Kac-Moody group over a field of characteristic $c$ with standard parabolic subgroups $P_I(\FF)$ and $G_I(\FF)$ the reductive, Lie Levi component subgroups for all $I\in \cS$.  The canonical map induced by inclusions of subgroups
\beq
\hcl_{I\in \cS}\BGI(\FF) \stackrel{\stackrel{q}{\thicksim}}{\longrightarrow} BK(\FF)
\label{hocoBGq}
\eeq
\nod is a $q$--equivalence for any prime $q \neq c$ and a rational equivalence for $c>0$.
\label{hocoBGqfinite}
\end{thm}

\begin{proof}
In this proof all groups mentioned are over a fixed $\FF$ that is suppressed in the notation.   Consider the fibration sequence
\beq
BU_I \longrightarrow \BPI {\longrightarrow} \BGI
\label{hocoBGqfib}
\eeq
\nod arising from the semidirect product decomposition $P_I \cong G_I \ltimes U_I$ (\ref{levidecomp}). By Theorem \ref{vanishthm}, $BU_I(\FF)$ is a $\Fq$--homology point (or $\QQ$--homology point for $q=0$) for all $I \in S$.
 The Serre spectral sequence for the fibration (\ref{hocoBGqfib}) shows $B(P_I \stackrel{\pi}{\longrightarrow} G_I)$ induces a isomorphism on $\Fq$--homology.
The inclusion $G_I \stackrel{\iota}{\longrightarrow} P_I$ is a section of $\pi$ and  must induce the inverse of $H_\ast(B(P_I \stackrel{\pi}{\longrightarrow} G_I), \Fq)$ on homology by naturality.
 Recalling (\ref{hocoBG}) the natural maps induced by inclusion of subgroups
\beq
\hcl_{I\in \cS}\BGI \stackrel{\stackrel{q}{\thicksim}}{\longrightarrow} \hcl_{I\in \cS}\BPI \stackrel{\thicksim}{\longrightarrow} BK
\label{hocoBGqcompose}
\nonumber
\eeq
\nod compose to yield the desired $q$--equivalence.
\end{proof}
\begin{proof}[Proof of Theorem \ref{KMbar}]
In the case at hand, Theorem \ref{hocoBGqfinite} 
 implies
\beq
\hcl_{I\in \cS}\BGI(\Fbar) \stackrel{\stackrel{q}{\thicksim}}{\longrightarrow}  BK(\Fbar)
\label{hocoBGqbar}
\eeq
\nod  where $\Fbar$ is the algebraic closure of ${\Fp}$.
Referring back to the construction Friedlander and Mislin used to produce Theorem \ref{KMbar} for reductive $G$ \cite[Theorem 1.4]{Fbar}, the maps $\BGIfbar {\rightarrow} \BGI$ are induced by the zig-zag of groups
\beq
G_I(\Fbar) {\longleftarrow} G_I(\Witt(\Fp)) {\longrightarrow} G_I(\CC) {\longrightarrow} G_I
\label{zzgroups}
\eeq
\nod where $\Witt(\Fp) \rightarrow \CC$ is a \emph{fixed} choice of embedding of the Witt vectors of  $\Fp$ into $\CC$ and $G_I(R)$ denotes the discrete group over $R$ of the same type as the (topological) complex reductive Lie group $G_I$.  Moreover, the maps of (\ref{zzgroups}) are natural with respect to the maps of group functors $G_I(-) \hookrightarrow G_J(-)$. Taking classifying spaces, we have compatible maps
\beq
 BG_I(\Fbar) {\longleftarrow}  BG_I(\Witt(\Fbar)) {\longrightarrow}  BG_I(\CC) {\longrightarrow} BG_I
 \nonumber
\label{zzclass}
\eeq
\nod
which are all $q$--equivalences.  Localizing at a prime $q$ distinct from $p$, we obtain compatible $\BGIfbarq \stackrel{\stackrel{q}{\thicksim}}{\rightarrow} \BGIq$.  Compatible   $\BGIfbar {\rightarrow} \BGI$ are then
  produced via an arithmetic fibre square \cite{arth}
  and induce a $q$--equivalence
\beq
 \hcl_{I\in \cS}\BGI(\Fbar) \stackrel{\stackrel{q}{\thicksim}}{\rightarrow} \hcl_{I\in \cS}\BGI.
\label{hocoBGqbarfried}
\eeq
\nod
Thus, equations (\ref{hocoBG}--\ref{hocoBGcomplex}),  (\ref{hocoBGq}) and (\ref{hocoBGqbarfried}) induce
\beq
BK(\Fbar) &\stackrel{\thicksim}{\leftarrow}& \hcl_{I\in \cS} BP_I(\Fbar) \stackrel{\stackrel{q}{\thicksim}}{\rightleftarrows} \hcl_{I\in \cS}\BGI(\Fbar) \nonumber \\
&\stackrel{\stackrel{q}{\thicksim}}{\rightarrow}& \hcl_{I\in \cS}\BGI \stackrel{\thicksim}{\rightarrow} BK .
\nonumber
\label{hocoBGqbarThm}
\eeq
\nod Choosing a fixed homotopy inverse for  the arrow pointing to the left,  we obtain the desired map.
\end{proof}
\begin{rem}
 Theorem \ref{KMbar} for reductive $G$ is one instance of the Friedlander--Milnor conjecture \cite{FriMil} in which $BG(\Fbar)$ approximates $BG$ homologically.
  More generally, natural discrete approximations  of $BG$ by $BG(\FF)$ can be extended to discrete approximations of Kac-Moody groups via Theorem \ref{hocoBGqfinite} as in the proof of Theorem \ref{KMbar}.
For example, Morel's \cites{Morel} confirmation of the Friedlander--Milnor conjecture for specific $G$ of small rank extends to
homological approximations of $BK$ by $BK(\FF)$ for any separably closed field $\FF$ whenever $K$ has a cofinal subposet ${\bf C}$ of $\cS$  with the property that $I \in {\bf C}$
implies  $G_I=H \ltimes T $ for some $T=(\CC, \times)^{k}$ and $H \in\{ SL_3, SL_4, SO_5, G_2\}$.
\label{FMc}
\end{rem}

\subsection{Unstable Adams operations for Kac-Moody groups}
\label{sec:compare}

For $q\neq p$, we will construct a local unstable Adams operation $\psk : \BKq {\rightarrow} \BKq$ compatible with the Frobenius map.  When $W$ has no element of order $p$, we can
assemble the local Adams maps via the arithmetic fibre square to obtain a global unstable Adams operation $\psk : BK {\rightarrow} BK$.

\begin{proof}[Proof of Theorem \ref{pskthm}]
 Let us first construct $\psk : \BKq {\rightarrow} \BKq$ for $q\neq p$ by noting that, localizing of the map (\ref{hocoBGqbarfried}),  we have homotopy equivalences
\beq
(\hcl_{I\in \cS}\BGIfbarq)\qcp \stackrel{\thicksim}{\longrightarrow} (\hcl_{I\in \cS}\BGIq)\qcp \stackrel{\thicksim}{\longrightarrow} \BKq .
\nonumber
\label{hocoBGlocal}
\eeq
\nod  As in the proof of Theorem  \ref{KMbar}, we have compatible $\BGIfbarq \stackrel{\thicksim}{\rightarrow} \BGIq$ induced by the zig-zag of groups
\beq
G_I(\Fbar) {\longleftarrow} G_I(\Witt(\Fp)) {\longrightarrow} G_I(\CC) {\longrightarrow} G_I
\nonumber
\label{zzgroups2}
\eeq
\nod where $\Witt(\Fp) \rightarrow \CC$ is a fixed embedding of the Witt vectors of  $\Fp$ into $\CC$ and $G_I(R)$
 denotes the discrete algebraic group over $R$ of the same type as the (topological) complex reductive Lie group $G_I$.
   From the naturality of the group functors, we have explicit topological models so that
\beq
\xymatrix{
\BGIfbar \ar[d]^{B(i)} \ar[r]^{BG_I(\varphi^k)}&  \BGIfbar \ar[d]^{B(i)}  \\
\BGJfbar \ar[r]^{BG_I(\varphi^k)}& \BGJfbar
\nonumber
\label{actionholim}}
\eeq
\nod commutes for all $I \subset J$ in $\cS$.   Localizing, we have a map of diagrams and $\psk:=(\varphi^k)\qcp$ extends to $\hcl_{I\in \cS}\BGIfbarq \simeq BK(\Fbar)\qcp \simeq \BKq$.

When $W$ has no elements of order $p$, we will work one prime at a time and assemble the maps using an arithmetic fibre square.  We have already constructed $\psk : \BKq {\rightarrow} \BKq$ for $q\neq p$. At $p$, we have
a homotopy equivalence
\beq
BN\pcp \stackrel{\thicksim}{\longrightarrow} BK\pcp
\nonumber
\eeq
where $N$ is the normalizer of the maximal torus and  this map is induced by the inclusion of groups \cite{nituthesis}. Kumar \cite{kumar}*{6.1.8}  presents $N$ as being generated by $T$ and $\{s_1, \ldots s_n\}$  so that under the
projection  $\pi:N \rightarrow W$ $s_i$ maps to a standard generator of $W$. Thus, we can attempt to define $\theta: N \rightarrow N$ in terms of these generators and need only check that Kumar's relations are satisfied to
obtain a group homeomorphism.  To get an unstable Adams operation we choose $t \mapsto t^{p^k}$  and $s_i \mapsto s_i$. Note that here $p$ is odd and this is needed to verify Kumar's relations.

Because $BK$ is a simply connected $CW$--complex \cite{nitutkm}, it is given as a homotopy pullback
\beq
\xymatrix{
BK \simeq BK^{{\wedge}}_\ZZ \ar[d]^{(-)^{{\wedge}}_\QQ} \ar[r]^{\prod (-)^{{\wedge}}_q}&  \prod \BKq \ar[d]^{(-)^{{\wedge}}_\QQ}  \\
BK^{{\wedge}}_\QQ  \ar[r]^{\prod (-)^{{\wedge}}_q}& \prod (\BKq)^{{\wedge}}_\QQ
\label{fibresquare}}
\eeq
\nod known as the arithmetic fibre square where $(-)^{{\wedge}}_\ZZ$ and ${(-)^{{\wedge}}_\QQ}$ denote localization with respect to $\ZZ$ and $\QQ$ homology, respectively \cite{arth}.
Now, by (\ref{hocoBGcomplex}) we have homotopy equivalence
\beq
BK^{{\wedge}}_\QQ \simeq (\hcl_{I\in \cS} \BGI^{{\wedge}}_\QQ)^{{\wedge}}_\QQ \simeq (\hcl_{I\in \cS} \prod K(2m_i,\QQ))^{{\wedge}}_\QQ
\label{QQlocal}
\eeq
\nod where the $m_i$ vary for different $\BGI$  and $K(n, \QQ)$ denotes the $n^{th}$ Eilenberg-MacLane space \cite{basics}.
 The homomorphism  $t \mapsto t^{p^k}$ of $S^1$ induces compatible
self maps of the $K(2m_i, \QQ)\simeq B^{2m_i-1}(S^1)\ratcp$ appearing in (\ref{QQlocal}). We may now define the desired map with (\ref{fibresquare}).
\end{proof}

\begin{rem}
In general, twisted Adams operations beyond those constructed here are expected. See \cite{BKtoBK} for rank two examples.
For instance, whenever $W \cong (\ZZ/2\ZZ)^{\ast n}$ the constructions of \cite{BKtoBK} generalize, but one must take care to check compatibility for complicated
$W$.
\label{Adamsremark}
\end{rem}

\begin{quest}
Are the unstable Adams operations constructed in Theorem \ref{pskthm} unique, up to homotopy, among
maps that restrict to the self-map $B(t \mapsto t^p)$ of ${BT}$?
\label{questuni}
\end{quest}

Now, that we have $\psk : \BKq {\rightarrow} \BKq$ let us note that when comparing $(\BKq)^{h\psk}$ and $\BKfq$
there is a natural map
\beq
\BKfq \simeq (\hcl_{I\in \cS} \BGIfq )\qcp \longrightarrow (\BKq)^{h\psk}
\label{mainmap2}
\eeq
\nod arising  from the diagram $D:\ZZ \times \cS \rightarrow Spaces$  via $ ( \bullet, W_I) \mapsto \BGIq$ on objects
and $(n, W_I \hookrightarrow W_J) \mapsto (\psk)^n B(G_I \hookrightarrow G_J)=B(G_I \hookrightarrow G_J)(\psk)^n $ on morphisms.  In particular, (\ref{mainmap2}) is the localization of the canonical
\beq
\hcl_\cS (\hl_\ZZ \BGIfq ) \longrightarrow  \hl_\ZZ (\hcl_{\cS} \BGIfq ).
\label{mainmap2canon}
\eeq
Generally, we do not expect homotopy limits and colimits to commute.  Our calculations in the next section that show that they rarely do in rank two examples.

\begin{quest}
What is the structure of the homotopy fibre of (\ref{mainmap2canon})?
\label{questmainmain}
\end{quest}

\section{Cohomology Calculations}
\label{sec:compmain}

Theorem \ref{hocoBGqfinite} allows us to study $BK(\Fpk)$ 
 in terms of {\em finite} reductive algebraic group classifying spaces, $BG_I(\Fpk)$ for $I \in \cS$, after localizing at a prime $q$ distinct from $p$.
Furthermore, \cite{Friedbook} reduces the study of $BG_I(\Fpk)\qcp$ to understanding homotopy fixed points $\BGIqfix$ under stable Adams operations, $\psk_I$. Thus, in principle, only the cohomology of compact Lie group classifying spaces is need as input data to compute $H^{\ast}(BK(\Fpk), \Fq)$.

In this section, we begin such calculations. Recent work by Kishimoto and Kono \cite{twisted} facilitates the determination the induced maps on cohomology  between $\BGIqfix$ as $I$ varies.
We then compare the results with $H^{\ast}(\BKqfix, \Fq)$ for rank two, infinite  dimensional  Kac-Moody groups.
  We close with explicit calculations of $H^\ast(BU^+(\Fp),\Fp)$
in specific cases where $W \cong (\ZZ/2\ZZ)^{\ast n}$.

\subsection{Restriction to Rank 2}
\label{subsec:r2}

To compare $(\BKq)^{h\psk}$ and $\BKfq$, we will partially compute their $\Fq$--cohomology rings.  These computations will become more tractable by restricting to simply connected rank 2 Kac-Moody
groups and $q$ odd. Notably $H^{\ast}(BK, \Fq)$ and its $\psk$--action can be determined explicitly. In the rank 2, non-Lie, case $\BKfq \simeq \hcl_{I\in \cS} (\BGIq)^{h\psk_I}$ is simply a homotopy pushout.  To compute cohomology, we use the Mayer-Vietoris sequence.
  We will also restrict to $q$ odd, so that for $I\in \cS$
\beq
\HBGIq \cong \HBTq^{W_I}
\label{weylfix}
\eeq
\nod with the restriction map from $\HBGIq$ to $\HBTq$ inducing this isomorphism \cite{basics}. The determination of $\HBGIqfix$ for $I\in \cS$ will occur in \ref{sec:levical}; this subsection will investigate
$H^{\ast}(\BKqfix, \Fq)$.

In the simply connected rank 2  case,  we have
 \cite{nituthesis,rank2mv}
\beq
H^{\ast}(BK, \Fq) = \Fq[x_4, x_{2l}] \otimes \Lambda(x_{2l+1})
\label{rank2coho}
\eeq
\nod as a ring where $\Lambda$ denotes an exterior algebra and  $l:=l(\{a,b\},q)\ge 2$ is a positive integer depending on $q$ and the generalized Cartan matrix for $K$, i.e. some non-singular $2 \times 2$ matrix given by
\beq
A=\left[\begin{array}{cc} 2 & -a \\ -b & 2\end{array}\right]
\eeq
\nod for $a$ and $b$ positive integers such that $ab \ge 4$.  An explicit description of $l$ is given in \cite{nituthesis}.
Notice non-singularity implies $ab>4$ and each $A$ is associated with one simply connected K.
 Work in  \cite{rank2mv} will determine the map $\psk$ induces on $\Fq$--cohomology.

\begin{prop}
For $K$ a rank 2, infinite dimensional complex Kac-Moody group and $q$ odd, $\psk$ acts on $\HBKq$ (\ref{rank2coho}) via
$(x_4, x_{2l}, x_{2l+1}) \mapsto (p^{2k}x_4 , p^{lk}x_{2l}, p^{lk}x_{2l+1})$.
\label{kmacts}
\end{prop}

\begin{proof}
Because $\HBGIq$ is concentrated in even degrees for $I\in \cS$, the Mayer-Vietoris sequence
associated to the homotopy pushout presentation of $BK$ (\ref{hocoBGcomplex})
 reduces to an exact sequence
\beq
0 \rightarrow \HBGevenq &\rightarrow& \HBGoneq \oplus  \HBGtwoq \rightarrow \HBTq   \nonumber \\
&\rightarrow& \HBGoddq \rightarrow 0
\label{mvkm}
\eeq
where $\HBGevenq = \HBGoneq \cap \HBGtwoq= \Fq[x_4, x_{2l}]$ and $\HBGoddq = \langle x_{2l+1} \rangle\HBGevenq$ for $x_{2l+1}$ the image of a homogeneous degree $2l$
class under the connecting homomorphism \cite{rank2mv}.  The $k^{th}$ unstable Adams operation $\psk$ acts on $\HBTq$ via multiplication by $p^k$ on generators,
 and commutes with the restriction (\ref{weylfix}) $\HBGIq \rightarrow \HBTq$ \cite{lieclass}.
\end{proof}

  We are presently unable to fully compute $H^{\ast}(\BKqfix, \Fq)$ for all rank 2 Kac-Moody groups. However, there is sufficient information
at $\Etwo$ in the Eilenberg-Moore spectral sequence associated to $H^{\ast}(\BKqfix, \Fq)$ to determine that in most rank 2 cases $H^{\ast}(BK(\Fpk), \Fq) \neq H^{\ast}(\BKqfix, \Fq)$, in contrast to the Lie case.

\begin{thm}  Consider the Eilenberg-Moore spectral sequence (EMSS) for the homotopy pullback of the diagram
\beq
\xymatrix{
&\BKq\ar[d]^{\De}\\
\BKq \ar[r]^-{1 \times \psk \circ \De}&(BK \times BK)\qcp}
\label{Gpullback}
\eeq
\nod converging to $\HBKqfix$ where $\psk$ is the $k^{th}$ unstable Adams operation.  For $q$ odd,  $\Etwo$, as a $\Fq$--algebra, is given in Table \ref{tablekmfixedthm}.
\label{kmfixedthm}
\end{thm}

\begin{table}
\caption{ The $\Fq$--algebra structure for $\Etwo$ converging to $\HBKqfix$  with bidegrees $|x_i|=(-1, i+1)$, $|\gamma_n(x_i)|=(-n, (i+1)n)$, and $|s_i|=(0, i)$. }
\begin{center}
\begin{tabular}{ |l|p{4cm}|p{4cm}| }
\hline
 & $p^{2k} = 1 (\modd ~q)$ &  $p^{2k} \neq 1 (\modd ~q)$ \\
\hline
$p^{kl} = 1 (\modd ~q)$ & $ \Lambda (x_3, x_{2l-1}) \otimes \Gamma(x_{2l}) \otimes \Fq[s_4, s_{2l}] \otimes \Lambda(s_{2l+1})  $ & $ {\Lambda(x_{2l-1}) \otimes \Gamma(x_{2l})} \otimes { \Fq[s_{2l}] \otimes \Lambda(s_{2l+1})}$  \\
\hline
$p^{kl} \neq 1 (\modd ~q)$ & $ {\Lambda(x_{3})} \otimes \Fq[s_{4}] $ &   $\Fq$ \\
\hline
\end{tabular}
\end{center}
\label{tablekmfixedthm}
\end{table}
\begin{proof}
 The $\Etwo$--page for the cohomological EMSS for (\ref{Gpullback}) is given by
\beq
\Tor_{\HBKq \otimes \HBKq} (\HBKq, \HBKq)   \label{torbg} \\
 \cong \Tor_{\Fq[x_4, y_4, x_{2l}, y_{2l}]\otimes \Lambda(x_{2l+1}, y_{2l+1})} (\Fq[s_4, s_{2l}] \otimes \Lambda(s_{2l+1}), \Fq[t_4, t_{2l}] \otimes \Lambda(t_{2l+1})) \nonumber
\eeq
\nod Here left copy of $\HBKq$ a $\HBKq \otimes \HBKq$--module via ${1 \times \psk \circ \De}$ which is given by the ring homomorphism
$
(y_4,  y_{2l}, y_{2l+1}) \mapsto (p^{2k}s_4, p^{lk}s_{2l},  p^{lk}s_{2l+1})
$
 and  $x_i \mapsto t_i$ on generators by Proposition \ref{kmacts}. Of course, the $\HBKq \otimes \HBKq$--module structure on the right $\HBKq$ is given by $y_i, x_i \mapsto t_i$.

To simplify $\Etwo$ we will employ a change of ring isomorphism. 
\begin{propout} [Change of Rings \cite{users} p.  280]
Let $A$, $B$, and $C$ be $k$--algebras over a field $k$.  Then,
\beq
  Tor^n_A (M,N)= Tor^n_C (N,L)= 0
\eeq
\nod for all $n>0$ implies
\beq
  Tor^n_{A\otimes B} (M,N \otimes_C L) \cong Tor^n_{B\otimes C} (M \otimes_A N,L)
  \label{mc}
\eeq
 where $M$,  $N$ and $L$ have the appropriate module structures so that (\ref{mc}) is well defined.
 \label{cofr}
\end{propout}

Let us use this change of ring isomorphisms with
\beq
(A,B,C) &=& (\Fq[x_4 , x_{2l}]\otimes \Lambda(x_{2l+1}) , \Fq[y_4- x_4 , y_{2l}- x_{2l}] \otimes \Lambda(y_{2l+1}- x_{2l+1}), \Fq), \nonumber \\
(M,N,L) &=& ( \Fq[s_4 , s_{2l}] \otimes \Lambda( s_{2l+1}),\Fq[t_4 , t_{2l}]\otimes \Lambda(t_{2l+1}, \Fq ) .\nonumber
\eeq
\nod
  This  gives
\beq
\Etwo \cong \Tor_{\Fq[z_4, z_{2l}] \otimes \Lambda(z_{2l+1})}  (\Fq[s_4, s_{2l}] \otimes \Lambda(s_{2l+1}), \Fq)
\eeq
where $z_i= y_i-x_i$
 acts via $(z_4, z_{2l}, z_{2l-1}) \mapsto  ((p^{2k}-1) s_4,(p^{lk}-1)s_{2l}, (p^{lk}-1)s_{2l+1})$.  We may employ Proposition \ref{cofr} further if $p^{2k}-1$ or $p^{lk}-1$ is nonzero modulo $q$.  In this way, we obtain Table \ref{E2tor}.

\begin{table}
\caption{$\Etwo$ converging to $\HBKqfix$ in terms of $\Tor$ where $z_i= y_i-x_i$ for $y_i$ and $x_i$.}
\begin{center}
\begin{tabular}{ |l|l| }
\hline
$( mod ~q)$ & $\Etwo$   \\
\hline
$p^{kl} =1$ and $p^{2k}=1$ & $ \Tor_{\Fq[z_4, z_{2l}] \otimes \Lambda(z_{2l+1})} {(\Fq[s_4, s_{2l}] \otimes \Lambda(s_{2l+1}), \Fq) } $  \\
\hline
$p^{kl} \neq 1$ and $p^{2k} = 1$ & $ \Tor_{\Fq[z_{2l}] \otimes \Lambda(z_{2l+1})}{ ({\Fq[s_{2l}] \otimes \Lambda(s_{2l+1})}, \Fq) }$  \\
\hline
$p^{kl} = 1$ and $p^{2k}\neq1$  & $ \Tor_{\Fq[z_{4}]}{ ({\Fq[s_{4}]}, \Fq) }$ \\
\hline
$p^{kl} \neq 1$ and $p^{2k}\neq1$  &   $\Tor_{\Fq}(\Fq, \Fq) $\\
\hline
\end{tabular}
\end{center}
\label{E2tor}
\end{table}
For example, the $p^{kl} = 1 (\modd ~q)$  and $p^{2k} \neq 1 (\modd ~q)$ case is obtained via the choices
\beq
(A,B,C, M, N, L)=  (\Fq,  \Fq[z_{2l}] \otimes \Lambda(z_{2l+1}), \Fq[z_4], \Fq[s_{2l}] \otimes \Lambda(s_{2l+1}), \Fq[s_4],\Fq) \nonumber
\eeq
\nod and the other entries are obtained similarly.
 In Table \ref{E2tor}, all the modules are trivial over their  respective $\Fq$--algebras. 

For $q$ odd, all the $\Fq$--algebras in Table \ref{E2tor} are finitely generated free graded commutative.  Let $\Fq(X)$ denote a free graded commutative algebra on the graded set $X$.

We follow \cite{users}*{p.  258-260} to compute $\Etwo$.
If one resolves the trivial $\Fq(X)$--module $\Fq$ using the Koszul complex $\Omega(X)$,
then
\beq
\Tor_{\Fq(X)} (L, \Fq)  = H( \Omega(X)\otimes L, d_L)
\eeq
\nod where  $L$ is any $\Fq(X)$--module.
In particular,  $\Omega(X)$ is $\Lambda(s^{-1}X^{even}) \otimes \Gamma(s^{-1}X^{odd})$ where $\Gamma$ denotes the divided power algebra. Here $s^{-1}X^{even}$ is a set of odd degree generators obtained from the even degree elements of $X$ by decreasing the their degree by one and  $s^{-1}X^{odd}$ is defined analogously.
In all the cases at hand, $L$ is a trivial $\Fq(X)$--module which implies $d_L$ is zero. Thus, $\Etwo$ is given by $\Omega(X)\otimes L$ (see Table \ref{tablekmfixedthm}). As the EMSS is a spectral sequence of $\Fq$--algebras \cite{users}, we have computed $\Etwo$ as an algebra.
 \end{proof}

Note that from Table \ref{tablekmfixedthm} the EMSS computing $\HBKqfix$ collapses at $\Etwo$ if $p^{kl} \neq  1 (\modd ~q)$ for degree reasons.  However, if $p^{kl} =  1 (\modd ~q)$,
$d_r(\gamma_r(x_{2l}))$ is potentially nonzero.

\subsection{Levi component calculations and a proof of Theorem \ref{rank2thm}}
\label{sec:levical}

Note that for odd $q$ and $I=\{i\}$
\beq
\HBGIq \cong \HBTq^{r_i} = \Fq[p_2(x,y) , p_4(x,y)] = \Fq[s_2 , s_4]
\label{weylfixed}
\nonumber
\eeq
\nod with $r_i$ generating $W_I$ and  $p_j(x,y)$ homogenous polynomials of degree $j$ in the degree $2$ generators of $\HBTq$, i.e. $p_2(x,y)$ is linear in $x$ and $y$.  The action of $\psk$ is given by a simple scalar multiplication on generators and we can compute $\HBGIqfix$
for $I\in S= \{ \emptyset, \{1\}, \{2\} \}$.

\begin{table}
\caption{$\HBGoneqfix
\cong \HBGtwoqfix$ for rank 2 Kac-Moody groups of infinite type and $q$ odd.}
\begin{center}
\begin{tabular}{ |l|l|l| }
\hline
& $\HBGoneqfix
\cong \HBGtwoqfix$    \\
\hline
$p^{k} = 1 (\modd ~q)$ & $\Lambda(z_1, z_{3}) \otimes  \Fq[s_2, s_{4}]$ \\
\hline
$p^{k} = -1 (\modd ~q)$ &  $\Lambda( z_{3}) \otimes  \Fq[ s_{4}]$  \\
\hline
$p^{k} \neq \pm 1 (\modd ~q)$ &  $\Fq$ \\
\hline
\end{tabular}
\label{BGonefix}
\end{center}
\end{table}

\begin{thm}  The Eilenberg-Moore spectral sequence (EMSS) for the homotopy pullback
 computes $\HBGIqfix$ where $\psk$ is the $k^{th}$ unstable Adams operation.  For $q$ odd, we may compute $\HBGIqfix$ for $I\in S= \{ \emptyset, \{1\}, \{2\} \}$ and, up to isomorphism, it is given in Table \ref{BGonefix} and
\beq
\HBTqfix =  \Lambda(z_1, z_1') \otimes\Fq[s_2, s_2'] &\iff&  p^{k} = 1 (\modd ~q) {\rm ~~~and~~~ } \nonumber \\
 \Fq &\iff& p^{k} \neq 1 (\modd ~q) ~~.
 \label{BTfixed}
\eeq
\label{levifixedthm}
\end{thm}
 \begin{proof}
 Let us first compute $\HBGoneqfix\cong \HBGtwoqfix$.
If we employ the techniques in the proof of Theorem \ref{kmfixedthm}, we obtain Table \ref{E2BGone} in which all $\Fq(X)$--modules are trivial.
Kozsul resolutions give
Table \ref{BGonefix}, but here the spectral sequence collapses at $\Etwo$ for degree reasons. Equation (\ref{BTfixed}) is obtained similarly.
\begin{table}
\caption{$\Etwo$ collapsing to $\HBGoneqfix\cong \HBGtwoqfix$ for $z_i= x_i -x_i'$ the difference of the generators for two copies of $\HBGoneqfix\cong \HBGtwoqfix$.}
\begin{center}
\begin{tabular}{ |l|l| }
\hline
& $\Etwo$    \\
\hline
$p^{k} = ~1 (\modd ~q)$ & $\Tor_{\Fq[z_2, z_{4}]} (\Fq[s_2, s_4], \Fq) $ \\
\hline
$p^{k} = -1 (\modd ~q)$ &  $\Tor_{\Fq[ z_4]}(  \Fq[s_4], \Fq)$  \\
\hline
$p^{k} \neq \pm 1 (\modd ~q)$ &   $\Tor_{\Fq}(  \Fq, \Fq)$ \\
\hline
\end{tabular}
\label{E2BGone}
\end{center}
\end{table}
\end{proof}

\begin{proof}[Proof of Theorem \ref{rank2thm}]
  Let us set notation for the Mayer-Vietoris sequence for computing $H^\ast (\hcl_\cS \BGIqfix, \Fq)$
\beq
\cdots \stackrel{\partial_{n-1}}{\rightarrow} H^n(\hcl_\cS \BGIqfix, \Fq) &\stackrel{\rho_n}{\rightarrow}&  H^n ( \BGoneqfix, \Fq) \oplus H^n ( \BGtwoqfix, \Fq)  \nonumber \\
&\stackrel{\De_n}{\rightarrow}& H^n ( \BTqfix, \Fq) \stackrel{\partial_n}{\rightarrow}  \cdots .
\label{mvlabeled}
\eeq
\nod We will proceed by cases.

{\bf Case 1:} $p^{k}\neq 1 (\modd ~q)$.  Here the vanishing of $\HBTqfix$ gives the second and third rows of Table \ref{colimoffixed} by (\ref{BTfixed}--\ref{mvlabeled}).
If $p^{k} \neq \pm 1 (\modd ~q)$, then the Theorem \ref{rank2thm} clearly holds.
If instead $p^{k} = -1 (\modd ~q)$, then  comparing
Tables \ref{tablekmfixedthm} and \ref{colimoffixed}  in degrees $3$ and $5$
  gives $\HBKqfix \neq \HBKqfixhoco$. 

{\bf Case 2:} $p^{k} = 1 (\modd ~q)$. In this case, $\HBGIqfix$ is isomorphic to the $\Fq$--cohomology of free loops on $BG_I$ all for $I\in \cS$ \cite{twisted}.  Furthermore, we know $\Delta$
 from the Mayer-Vietoris sequence (\ref{mvlabeled}) almost completely by loc. cit.
 A careful count of dimensions will show
 $H^{4l}(\hcl_{I\in \cS} {(BG_I)}^{h\psk})^{\wedge}_q, \Fq) \neq H^{4l}(\BKqfix , \Fq)$.  Let $|\cdot|$ denote the rank of a vector space for this purpose.

On $\Etwo$ of the EMSS for $H^{\ast }(\BKqfix , \Fq)$  total degree $4l$
has rank 6 for $l$ odd and rank 8 for $l$ even.
  Because this is a spectral sequence of algebras, $d_r$ of all these generators vanish
since they are products of permanent cocycles for degree reasons noting that $q \neq 2$ implies $\gamma_2(x_{2l})=\frac{\gamma_1(x_{2l})}{2}^2$.  Indeed, they are all permanent cocycles as they are not the target of any differential for degree reasons and we recover
$|H^{4l}(\BKqfix , \Fq)|$.

Considering (\ref{BTfixed}) and Table \ref{BGonefix} we see that
\beq
|H^{4l-1}( \BTqfix , \Fq))|=2 \cdot |({BG_i}^{\wedge}_q)^{h\psk}, \Fq)|=4l
\nonumber
\eeq
\nod for $i \in \{1,2 \}$.  It follows that
\beq |H^{4l}(\hcl_{I\in \cS} \BGIqfix, \Fq)|= |\ke(\Delta_{4l-1}) | + |\ke(\Delta_{4l})| \nonumber \eeq
\nod by elementary homological methods. By \cites{rank2mv,twisted}, there are isomorphisms of graded abelian groups
\beq \ke(\Delta) \cong \HBTqfix^W \cong \Fq[s_4, s_{2l}] \otimes \Lambda(x_3, x_{2l-1})
\nonumber
\label{weylactcoho}
\eeq
\nod where $W$ is the infinite dihedral Weyl group of $K$ which acts on $\HBTqfix \cong \Lambda(x_1, x_1') \otimes\Fq[s_2, s_2']$.
Here we extend the standard action on  $H^{\ast}( BT, \Fq) \cong \Fq[s_2, s_2']$
via the identification $ds_i=x_i$ so that $d$ commutes with the action and satisfies the Leibnitz rule. In the terminology of \cites{twisted}, $\BTqfix$ is a twisted loop space  with associated derivation $d$ and $d$ commutes with restriction.
Considering  $\Fq[s_4, s_{2l}] \otimes \Lambda(x_3, x_{2l-1})$ in degree $4l$ and $4l-1$
gives that $H^{4l}(\hcl_{I\in \cS} {(BG_I)}^{h\psk})^{\wedge}_q, \Fq)$ has rank 5 for $l$ odd and rank 6 for $l$ even. Hence
 $p^{k} = 1 (\modd ~q)$ implies $\HBKqfix$ and $\HBKqfixhoco$ are distinct (and nontrivial) for odd $q$.
 \end{proof}

Table \ref{colimoffixed} summarizes  our deductions about the structure of
\beq
H^\ast (\hcl_\cS \BGIqfix, \Fq) \cong H^\ast ( BK(\Fpk), \Fq)
\nonumber
\eeq \nod within the proof of Theorem \ref{rank2thm}.
The methods of \cite{rank2mv}, where (\ref{mvkm}) is used to compute $\HBKq$, and a good understanding of the derivation associated to a twisted loop space should in principle allow further computation
in the $p^{k} = 1 (\modd ~q)$ case.  We preform no such calculations here but note that \cite[Theorem 8.3]{AR} computes $H^\ast ( BK(\Fpk), \Fq)$ in most of these interesting cases.

\begin{table}
\caption{$H^\ast ( BK(\Fpk), \Fq)$ for rank 2 Kac-Moody groups of infinite type.}
\begin{center}
\begin{tabular}{ |l|l| }
\hline
& $H^\ast ( BK(\Fpk), \Fq)$    \\
\hline
$p^{k} = ~1 (\modd ~q)$ & $H^\ast (\hcl_\cS L\BGI, \Fq)$ \\
\hline
$p^{k} = -1 (\modd ~q)$ &  $\Lambda( z_{3}) \otimes  \Fq[ s_{4}] \oplus \Lambda( z_{3}') \otimes  \Fq[ s_{4}']$  \\
\hline
$p^{k} \neq \pm 1 (\modd ~q)$ & $\Fq $  \\
\hline
\end{tabular}
\label{colimoffixed}
\end{center}
\end{table}

\subsection{Examples of $H^{\ast}(BU^+(\Fpk), \Fp)$}
\label{sec:calcp}

  In this subsection, we preform some preliminary calculations for the group cohomology of $U^+(\Fpk)$ at $p$ using our new homotopy decomposition (\ref{hocoUnewslick}) and  briefly discuss considerations in the general case.
  In our examples, we will see that $H^\ast (BU^+(\Fpk), \Fp)$ is infinitely generated as a ring.
    Along the way, we give
  explicit descriptions of the groups $U^+(R)$ underlying our cohomology calculations.  Unless otherwise stated, when we refer to Kac-Moody groups in this subsection we mean some image of Tits's \cite{TitsKM} explicit
  (discrete, minimal, split)  Kac-Moody group functor $K(-): {\bf Rings} \rightarrow {\bf Groups}$ from the category of commutative  rings with unit to the category of groups.  Likewise, $U^+(R)$ will be Tits's subgroup generated by the positive (real) root groups (\ref{U+def}).
  When $R$ is clear from context, we will simply write $K$ or $U^+$.

Our simplest (and most explicit) calculations will be for $U_2:=U_2^+(R) \le K_2(R)$  the positive unipotent subgroup of the infinite dimensional Kac-Moody group $K_2(R)$ with generalized Cartan matrix
  \beq
\left[\begin{array}{cc} 2 & -a \\ -b & 2\end{array}\right]
\label{gcm2}
\eeq \nod where $a, b \ge 2$.  In this case the relevant diagram  of unipotent subgroups of $U_2^+(R)$ indexed over $W$ is
\beq
\xymatrix{
\ldots   U_{tst}  & U_{st}\ar[l] & U_{t} \ar[l]& e \ar[l] \ar[r]& U_s \ar[r]& U_{st} \ar[r] & U_{sts} \ldots
}
\label{dubtel}
\eeq
\nod as here $W$ is the infinite dihedral group.  When $a=b=2$ and $R$ is a field, we have an affine Kac-Moody group $K_2(R)\cong GL_2(R[t,t^{-1}])$ and $U_2^+$ is known explicitly (cf. \cite{Titscover}).
We will identify $U_w(R)$ directly from the presentation of Tits \cite{TitsKM}
 and recover $U_2^+(R)$ from the colimit presentation induced by (\ref{hocoUnewslick}).  In fact, for any Kac-Moody group with generalized Cartan matrix $A=(a_{ij})_{1 \le i,j \le k}$ such that $|a_{ij}|\ge 2$ for all $1 \le i,j \le k$
the finite dimensional unipotent subgroups $U_w(R)$ have a very simple form.

\begin{lem} Let $K(R)$ be the discrete Kac-Moody group over $R$
with generalized Cartan matrix $A=(a_{ij})_{1 \le i,j \le k}$ such that $|a_{ij}|\ge 2$ for all $1 \le i,j \le k$. Then, for $U_w(R)\le K(R)$ defined by (\ref{uwdef}) we have
\beq
U_w(R) \cong (R,+)^{ \oplus l(w)}
\nonumber
\eeq
\nod such that the image of $U_w(R) \hookrightarrow U_{ws}(R)$ for $l(w) + 1 = l(ws)$ is the first  $l(w)$ factors.
\label{projlemma}
\end{lem}

\begin{proof}
  Recall that $w^{-1}\Phi^{+} \cap \Phi^{-} = \Theta_w:=\{ {\alpha_{i_1}}, \alpha_{i_2}, \ldots, {ws_{i_l}\alpha_{i_l} }\}$ and  that the multiplication map
\beq
U_{\alpha_{i_1}} \times U_{s_{i_1}\alpha_{i_2}} \times \ldots \times U_{ws_{i_l}\alpha_{i_l} } \stackrel{m}{\longrightarrow} U_w
\nonumber
\label{multiply2}
\eeq
\nod is an isomorphism of sets (\ref{multiply}).
 It is immediate from the presentation of Tits \cite[3.6]{TitsKM} that for $\alpha, \beta \in \Phi^{+}$,
$\{\NN \alpha + \NN \beta \} \cap \Phi^{+} = \{ \alpha , \beta\}$
 implies
 $ \langle U_\alpha ,  U_\beta \rangle = U_\alpha\times U_\beta$.  Moreover, the $W$ of action on $\Phi$ guarantees that
 $\{\NN \alpha +\NN \beta \} \cap \Phi^{+} \subset \Theta_w$ whenever $\alpha, \beta \in \Theta_w$.
 Direct computation in the hypothesized cases of the action of $W$ on the simple roots (\ref{actiondef2}) shows that each of the positive coefficients of $ws_{i_l}\alpha_{i_l}$ expressed in terms of the simple roots is maximal
 and one coefficient is strictly largest  within $\Theta_w$.
  In particular, if $\alpha \in \Theta_w$ is distinct from $ws_{i_l}\alpha_{i_l}$ then $U_\alpha$ and $U_{ws_{i_l}\alpha_{i_l}} $ commute because the coefficients demand that if $m_1 \alpha + m_2 ws_{i_l}\alpha_{i_l} \in \Theta_w$, then  $(m_1, m_2)$ is  $(0, 1)$ or $(m_1, 0)$. However, $m_1 \alpha \in \Phi^{+} \implies m_1=1$ \cite{kumar}.

 Since $u ,v \in U_w$ implies
\beq
uv=(u_{{1}} u_{2}  \ldots  u_{ l-1 })u_{ {l} } (v_{{1}} v_{{2}}  \ldots v_{ l-1 } )v_{ l }= ( u_{1} u_{2}  \ldots  u_{ l-1 })(v_{1} v_2  \ldots v_{ l-1 } )u_{ l} v_{ l} \nonumber
\eeq
\nod for $u_j, v_j \in U_{s_{i_1} \ldots s_{i_{j-1}}\alpha_{j}}$, we have $U_w \cong U_{ws_{i_l}} \times U_{\gamma}$ for $\gamma =ws_{i_l}\alpha_{i_l}$ .  The lemma follows by induction on $l(w)$ and the fact that all root subgroups are isomorphic to $(R, +)$.
\end{proof}

Returning to $U_2(R)$ and considering  the right telescope of (\ref{dubtel}), we obtain a projective limit of graded abelian groups upon taking group cohomology 
\beq
\xymatrix{
 H^{\ast}(e, \Fp)& H^{\ast}(U_s, \Fp)  \ar@{->>}[l]& H^{\ast}(U_{st}, \Fp)  \ar@{->>}[l] &  H^{\ast}(U_{sts}, \Fp)  \ldots  \ar@{->>}[l]
}
\label{righttele}
\eeq
\nod which is a sequence of surjections by the Kunneth theorem.  In particular,
\beq
\lim_{\leftarrow} {}^1 H^{\ast}(U_w(R), \Fp) = 0
\label{lim1=0}
\eeq
\nod and the cohomology spectral sequence of the mapping telescope of classifying spaces \cite{users} collapses at $\Etwo$.

There is some subtly in defining the positive unipotent subgroup of a Kac-Moody group over an arbitrary ring (cf. Remark 3.10 (d) in \cite{TitsKM}) so we will state the rest of our results in this section only over fields.

\begin{thm} Let $U_2:=U_2^+(\LL) \le K_2(\LL)$ be the positive unipotent subgroup of the Kac-Moody group $K_2(\LL)$ over a field $\LL$ with generalized Cartan matrix given by (\ref{gcm2}). Then
we may compute the group cohomology of $U_2(\LL)$ as
\beq
H^{\ast}(BU_2(\LL), \FF) = \bigotimes_{\ZZ>0} H^{\ast}(\LL, \FF) \oplus \bigotimes_{\ZZ>0} H^{\ast}(\LL, \FF)
\nonumber
\eeq
\nod where $\LL$ is considered as an abelian group and $\FF$ is a field.
\label{affinethm}
\end{thm}
\begin{proof}
 As $\lim_{\cat{W}^{op}}^1 H^{\ast}(U_w(\LL), \FF) = 0$ in this case (\ref{lim1=0}), the spectral sequence of the homotopy colimit of classifying spaces collapses at $\Etwo$ and computes $H^{\ast}(BU^+(\LL), \FF)$ as the projective limit of graded abelian groups $\lim_{\cat{W}^{op}} H^{\ast}(U_w(\LL), \FF)$.  Working one telescope at a time, we see the limit of graded abelian groups (\ref{righttele}) is just
$\lim_{\ZZ>0} H^{\ast}((\LL)^{n}, \FF)$ by Lemma \ref{projlemma}.  This limit is computed in each gradation  so that, unlike the limit of underlying abelian groups, only finite tensor words are possible
in each gradation.
  The left telescope is computed in the same way.  As the double telescope is the one point union of the left and right telescopes, $H^{\ast}(BU_2(\LL), \FF)$ is simply the direct sum of their $\FF$--cohomologies.
  \end{proof}

For the case of interest, $H^{\ast}(BU^+(\Fpk), \Fp)$, cohomology classes are represented by finite tensor words in the generators of $H^{\ast}(\Fpk, \Fp)$.  For instance, for $k=1$ we have
\beq
H^{\ast}((\ZZ / p \ZZ)^n, \Fp)= \left\{
{
\begin{array}{lr}
\FF_2[x_1, \ldots, x_n] &  p=2  \\
\Lambda[x_1, \ldots, x_n] \otimes \Fp[y_1, \ldots, y_n]&  p ~{\rm odd}
 \end{array}
}
\right .
\nonumber
\eeq
\nod where $|x_i|=1$ and $|y_i|=2$.  Applying the double telescope construction
\beq
H^{\ast}(BU^+(\ZZ / 2 \ZZ), \FF_2) &=& (\lim_{\leftarrow}  \FF_2[x_1, \ldots, x_n] )^{\oplus 2}  \nonumber \\
&=&  \FF_2[x_1, ,  \ldots, x_n, \ldots ] \oplus \FF_2[ \overline{x}_1, \ldots, \overline{x}_n, \ldots ].
\nonumber
\eeq
For $p$ odd
\beq
H^{\ast}(BU^+(\ZZ / p \ZZ), \Fp) &=& (\lim_{\leftarrow}  \Lambda[x_1, \ldots, x_n] \otimes \Fp[y_1, \ldots, y_n])^{\oplus 2}   \nonumber \\
  &=&   \Lambda[x_1, \ldots, x_n, \ldots] \otimes \Fp[y_1, \ldots, y_n, \ldots] \nonumber \\
   & \oplus &   \Lambda[\overline{x}_1, \ldots, \overline{x}_n, \ldots] \otimes
    \Fp[\overline{y}_1, \ldots, \overline{y}_n, \ldots]
  .
  \nonumber
\eeq

 Lemma \ref{projlemma} lets us  determine the group structure of more general $U^+(R)$ which have an interesting presentation. Briefly, when Lemma \ref{projlemma} applies we have
\beq
U^+(R)= \cl_{w\in \cat{W}} R^{l(w)}
\nonumber
\eeq
\nod where all maps are inclusions $R^{l(v)} \hookrightarrow R^{l(w)}$ with image the first ${l(v)}$ factors.  This leads to the presentation
\beq
U^+(R)= \langle U_\alpha |    u_\alpha u_\beta u_\alpha^{-1} u_\beta^{-1}=1, ~ { \rm if } ~ \exists w ~ { \rm with} ~ w \{ \alpha, \beta \} \subset \Phi^- \rangle
\nonumber
\eeq
 \nod where $w \in W$, $u_\gamma$ is an arbitrary element of  $U_\gamma$ and $\gamma \in \Phi^+$.  Pictorially, there is a generating $U_\alpha \cong (R, +)$
 for each node in Hasse diagram of the poset $W$, i.e. the directed graph with vertices and edges $(W,\{ (w, ws)| l(w)+1= l(ws)  \})$ (see, e.g., Figure \ref{pic:treeweyl}).
 Generating groups $U_\alpha$ and $U_\beta$ commute exactly when there is a (unique) path between them in this directed graph.  If there is no path between the  corresponding nodes,  $U_\alpha \ast U_\beta$ embeds into $U^+$.

 For example, if $K(R)$  has rank $3$  with generalized Cartan matrix such as
 \beq
\left[\begin{array}{ccc} 2 & -2 &-3\\  -2 & 2 & -2 \\ -2 & -4 & 2\end{array}\right] ,
\label{gcmrank3}
\eeq
\nod then the poset $\cat{W}$  is generated by the directed valence--3 tree pictured in Figure \ref{pic:treeweyl}. Two elements commute exactly when they are both contained in some $U_w(R)$.

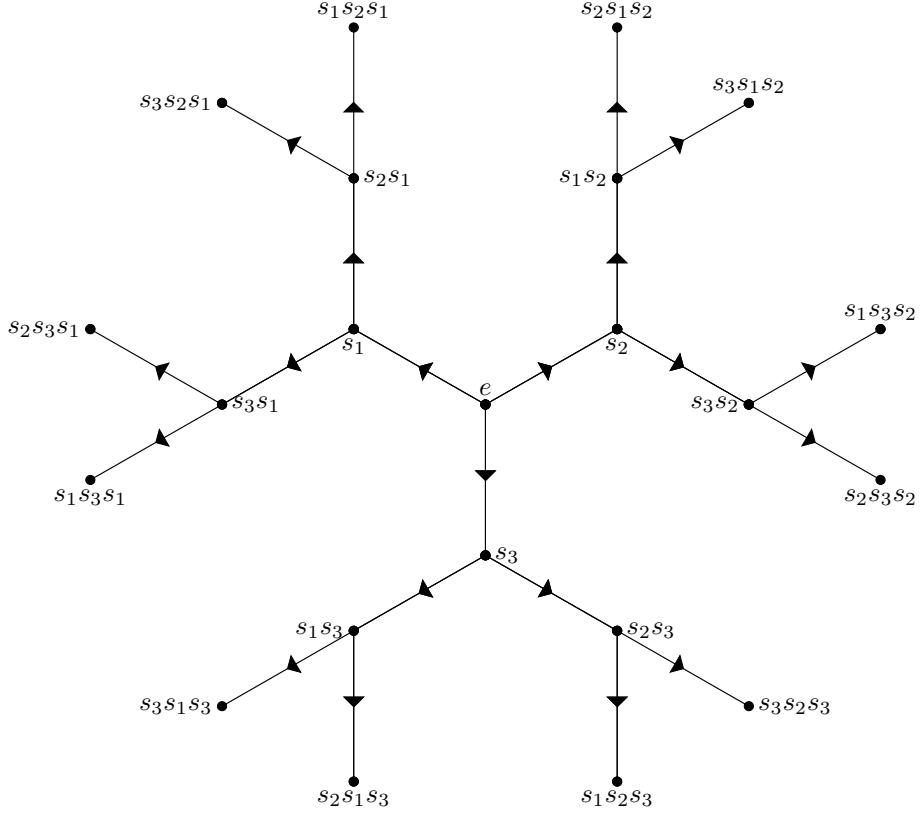
\begin{figure}
\begin{tikzpicture}[scale=1]
\def\initalrotate{0}
\def\hexa#1#2{
   \draw[fill=black] #1 #2 +(150+\initalrotate:2) circle (.6mm)
 -- node[rotate=90+\initalrotate] {\midarrow} +(210+\initalrotate:2) circle (.6mm)
 -- node[rotate=150+\initalrotate] {\midarrow} +(270+\initalrotate:2)  circle (.6mm)
 -- node[rotate=30+\initalrotate]{\midarrow} +(330+\initalrotate:2) circle (.6mm)
 -- node[rotate=90+\initalrotate] {\midarrow} +(390+\initalrotate:2) circle (.6mm);}

 \def\linedota#1#2{
   \draw[fill=black] #1 #2 ++(30+\initalrotate:2)
 -- node[rotate=90+\initalrotate] {\midarrow} +(90+\initalrotate:2) circle (.6mm) node[anchor=270] {$s_2s_1s_2$};
   \draw [fill=black]#1 #2 ++(150+\initalrotate:2)
 -- node[rotate=90+\initalrotate] {\midarrow} +(90+\initalrotate:2) circle (.6mm) node[anchor=270] {$s_1s_2s_1$};
 }

  \def\linedotb#1#2{
   \draw[fill=black] #1 #2 ++(30+\initalrotate:2)
 -- node[rotate=90+\initalrotate] {\midarrow} +(90+\initalrotate:2) circle (.6mm) node[anchor=180] {$s_3s_2s_3$};
   \draw [fill=black]#1 #2 ++(150+\initalrotate:2)
 -- node[rotate=90+\initalrotate] {\midarrow} +(90+\initalrotate:2) circle (.6mm) node[anchor=90] {$s_2s_3s_2$};
 }

  \def\linedotc#1#2{
   \draw[fill=black] #1 #2 ++(30+\initalrotate:2)
 -- node[rotate=90+\initalrotate] {\midarrow} +(90+\initalrotate:2) circle (.6mm) node[anchor=90] {$s_1s_3s_1$};
   \draw [fill=black]#1 #2 ++(150+\initalrotate:2)
 -- node[rotate=90+\initalrotate] {\midarrow} +(90+\initalrotate:2) circle (.6mm) node[anchor=0] {$s_3s_1s_3$};
 }

   \def\linedotd#1#2{
   \draw[fill=black] #1 #2 ++(30+\initalrotate:2)
 -- node[rotate=90+\initalrotate] {\midarrow} +(90+\initalrotate:2) circle (.6mm) node[anchor=90] {$s_2s_1s_3$};
   \draw [fill=black]#1 #2 ++(150+\initalrotate:2)
 -- node[rotate=90+\initalrotate] {\midarrow} +(90+\initalrotate:2) circle (.6mm) node[anchor=90] {$s_1s_2s_3$};
 }

  \def\sixdot#1#2{
  \draw[fill=black] #1 #2 +(30:2) circle (.6mm) node[anchor=0]{$s_1s_2$}
   +(150:2) circle (.6mm)node[anchor=180]{$s_2s_1$}
   +(210:2) circle (.6mm)node[anchor=90]{$s_1$}
 +(270:2)  circle (.6mm)node[anchor=270]{$e$}
  +(330:2) circle (.6mm) node[anchor=90]{$s_2$}-- cycle;
  }

    \def\sixdottwo#1#2{
  \draw[fill=black] #1 #2 +(30:2)  node[anchor=270]{}
  +(90:2)  circle (.6mm) node[anchor=0]{$s_3s_2s_1$}
   +(210:2) circle (.6mm)node[anchor=0]{ $s_2s_3s_1$}
  ;
  }

\def\sixdotthree#1#2{
  \draw[fill=black] #1 #2 +(30:2) circle (.6mm) node[anchor=180]{}
  +(90:2) circle (.6mm)  node[anchor=0]{}
   +(150:2) circle (.6mm)node[anchor=180]{$s_3s_1$}
 +(270:2)  circle (.6mm)node[anchor=0]{$s_1s_3$}
  +(330:2) circle (.6mm) node[anchor=180]{$s_3$}
  ;
  }

 \def\sixdotfour#1#2{
  \draw[fill=black] #1 #2 +(30:2) circle (.6mm) node[anchor=0]{$s_3s_2$}
  +(90:2) circle (.6mm)  node[anchor=0]{}
 +(270:2)  circle (.6mm)node[anchor=180]{$s_2s_3$}
  ;
  }

    \def\sixdotfive#1#2{
  \draw[fill=black] #1 #2
  +(90:2) circle (.6mm)  node[anchor=270]{$s_3s_1s_2$}
   +(150:2) circle (.6mm)node[anchor=270]{}
   +(210:2) circle (.6mm)node[anchor=270]{}
 +(270:2)  circle (.6mm)node[anchor=270]{}
  +(330:2) circle (.6mm) node[anchor=270]{$s_1s_3s_2$}-- cycle;
  }

\def\distanceamongcells{2*1.732050}

\hexa{(0,0)}{(0,0)}
\linedota{(0,0)}{(0,0)}
\def\initalrotate{ 240}
\linedotb{(0,0)}{(-60:\distanceamongcells)}
\def\initalrotate{ 120}
\linedotc{(0,0)}{(240:\distanceamongcells)}
\def\initalrotate{ 180}
\linedotd{(0,0)}{(270:4)}

\def\initalrotate{300}
\foreach \a in {0,-60} {
\def\initalrotate{300 + \a}
\hexa{(0,0)}{(\a:\distanceamongcells)}
}

\foreach \a in {240} {
\def\initalrotate{240 + \a}
\hexa{(0,0)}{(\a:\distanceamongcells)}
}

\def\initalrotate{180}
\hexa{(0,0)}{(270:6)}
\def\initalrotate{60}
\hexa{(0,0)}{(180:\distanceamongcells)}

  \sixdot{(0,0)}{(0,0)}
  \sixdottwo{(0,0)}{(180:\distanceamongcells)}
  \sixdotthree{(0,0)}{(240:\distanceamongcells)}
  \sixdotfour{(0,0)}{(300:\distanceamongcells)}
  \sixdotfive{(0,0)}{(0:\distanceamongcells)}

\end{tikzpicture}
\caption{A picture of the graph underlying the poset $\cat{W}=\langle s_1, s_2, s_3 | s_i^2=1\rangle$ and the weak Bruhat order with elements of length at most 3 labeled.}
\label{pic:treeweyl}
\end{figure}

Starting with the homotopy decomposition (\ref{hocoUnewslick}) of Theorem \ref{hocoUnewthmslick}, it is straight forward to apply  Theorem \ref{coverholim} and Theorem \ref{svkthm} to observe
\beq
\cl_{w\in \cat{W} } R^{l(w)} \cong \cl_{m \in \NN} ( \cl_{\cat{W}_m} R^{l(w)} )
\nonumber
\eeq
\nod for $\NN$ the linearly ordered poset and ${\cat{W}_m}$ is the full subposet of $\cat{W}$ of elements of length at most. 
 Set  $U_{\cat{W}_m}:=\cl_{\cat{W}_m} U_w$. By Corollary \ref{combinlocal}
we have a commutative diagram
\beq
\xymatrix{
\hcl_{L^{-1}(\cat{W}_m)} U_{\cat{W}_m}/R^{l(w_I)} \ar[d] \ar[r]&   \hcl_{L^{-1}(\cat{W}_m)} B( R^{l(w_I)}) \ar[d]_{L}  \ar[dr] & \\
\hcl_{\cat{W}_m} U_{\cat{W}_m}/R^{l(w)} \ar[r]&  \hcl_{\cat{W}_m} B(R^{l(w)}) \ar[r]  & BU_{\cat{W}_m}
\label{3pullback}} \nonumber
\eeq
\nod whose vertical maps are weak equivalences.
 Since $\hcl_{L^{-1}(\cat{W}_m)} U_{\cat{W}_m}/R^{l(w_I)}$ can be modeled by a $1$--dimensional $CW$--complex and is simply connected by Corollary \ref{simcon}, it must be contractible so that
\beq
 B( \cl_{\cat{W}_m} R^{l(w)}) \simeq ( \hcl_{\cat{W}_m} B(R^{l(w)})).
 \label{Nfilteredcolimit}
\eeq
Thus, we have homotopy equivalences
\beq
BU^+(R) &\simeq&  \hcl_{w\in \cat{W}} B(R^{l(w)}) \nonumber \\
 &\simeq& \hcl_{\NN} ( \hcl_{\cat{W}_m} B(R^{l(w)}))  \nonumber \\
 &\simeq& \hcl_{\NN} B( \cl_{\cat{W}_m} R^{l(w)})
\label{treeexpression}
\eeq
\nod induced by the obvious maps.  Cohomology calculations  for
$
 B( \cl_{\cat{W}_m} R^{l(w)}) \simeq  \hcl_{\cat{W}_m} B(R^{l(w)})
$ can be preformed by iterated Mayer-Vietoris calculations.
    For instance, for $K$  with the generalized Cartan matrix (\ref{gcmrank3}) considered above we have
    \beq
 \cl_{\cat{W}_m} R^{l(w)} &=& [R^{m} \ast_{R^{m-1}} R^{m} \ast_{R^{m-2 }} R^{m} \ldots  R^{m} \ast_{R^2 }R^{m} \nonumber \\
 &\ast_k& R^{m} \ast_{R^2} R^{m} \ast_{R^3 }R^{m}\ldots R^{m} \ast_{R^{m-1}} R^{m}  ]^{\ast 3}
 \label{amal}
\eeq
\nod and each $R^i$ of the amalgamated products denotes the first $i$ factors.  For example, this implies
\beq
H^{\ast}(\cl_{\cat{W}_3} (\ZZ / 2 \ZZ)^{l(w)} , \FF_2) = (\FF_2[x, x_0, x_1, x_{00}, x_{01}, x_{10}, x_{11}]/ I)^{\oplus 3}
\nonumber
\eeq
\nod when $m=3$ where
\beq
I= \langle x_0 x_1, x_0 \{ x_{10}, x_{11} \}, x_1 \{x_{00}, x_{01} \}, x_{00}\{ x_{10}, x_{10}, x_{11} \},
  x_{01}\{  x_{10}, x_{11} \}, x_{10} x_{11}\rangle \nonumber
\eeq
\nod as nonidentity nodes in the  directed graph underlying $\cat{W}_3$ (and pictured in Figure \ref{pic:treeweyl}) contribute generators and the multiplicative structure is determined by whether the corresponding root groups commute or generate a free product. The three direct summands correspond to the each of the three main branches of the tree underlying $\cat{W}_3$.  For instance, on the $s_2$ branch nodes correspond to generators via
$s_2 \leftrightarrow x$, $s_1s_2 \leftrightarrow x_0$, $ s_3s_2 \leftrightarrow x_1$,
$s_2s_1s_2 \leftrightarrow x_{00}$, $s_3s_1s_2 \leftrightarrow x_{01}$, $s_1s_3s_2 \leftrightarrow x_{10}$, and  $s_2s_3s_2 \leftrightarrow x_{11}$.
The other branches have a similar correspondence.
   This cohomology may be obtained inductively from the expression of $\cl_{\cat{W}_3} R^{l(w)}$ given in (\ref{amal}).
 More generally, this process gives in the following result.

\begin{thm} Let $K(\LL)$ be a discrete Kac-Moody over a field $\LL$
with generalized Cartan matrix $A=(a_{ij})_{1 \le i,j \le n}$ such that $|a_{ij}|\ge 2$ for all $1 \le i,j \le n$ with positive unipotent subgroup $U^+(\LL)$. Then
 the cohomology of $BU_+(\LL)$ with coefficients in any field $\FF$ is given as
\beq
H^{\ast}(BU^+(\LL), \FF)
        &=&\lim_{\leftarrow}  H^{\ast}( \hcl_{\cat{W}_m} B(\LL^{l(w)}), \FF) \nonumber \\
        &=& \otimes_{\cat{W}-\{e\}}  H^{\ast} (\LL , \FF) / I
        \label{treecohoeq}
        \nonumber
\eeq
\nod where  the ideal $I$ is generated by all products of pairs of elements in different copies of $H^{\ast} (\LL, \FF)$
corresponding to nonidentity nodes in the directed graph underlying the poset $\cat{W}$ that are disconnected.
\label{treecalc}
\end{thm}

For the rank 3 example discussed above (\ref{gcmrank3}),  $H^{\ast}(BU^+(\FF_2), \FF_2)$ is the $\FF_2$--algebra generated by the nodes (except the root) of an infinitely extended Figure \ref{pic:treeweyl} with relations $x y=0$ for  generators $x$ and $y$
whose nodes are not connected.

All the examples described here concern Kac-Moody groups with Weyl groups whose Hasse diagrams are trees so that (\ref{Nfilteredcolimit}--\ref{amal}) is available to clarify  inductive arguments
  and diagrams over $\cat{W}_m$ can be seen as iterated pushouts.
Because the cohomology spectral sequence calculation is trivialized for limits satisfying the Mittag--Leffler condition, we pose the following question.

\begin{quest} Let $K(\Fpk)$ be a discrete minimal Kac-Moody group over $\Fpk$ with unipotent $p$--subgroups $U_w(\Fpk)$ defined by (\ref{uwdef}). When does the diagram of group cohomologies
$D_K: \cat{W} \rightarrow \cat{Groups} \rightarrow \cat{Graded} ~\cat{Abelian} ~\cat{Groups}$ given by
$w \mapsto  U_w(\Fpk) \mapsto H^{\ast}(U_w(\Fpk), \Fp)$ satisfy the Mittag--Leffler condition?
\end{quest}

 In the tree shaped cases not covered here, a positive answer to this question would permit cohomology calculations as outlined in this section.
 Obviously, Lemma \ref{projlemma} greatly simplifies our calculations.  For example, for an infinite dimensional Kac-Moody group with generalized Cartan matrix
 \beq
\left[\begin{array}{cc} 2 & -a \\ -1 & 2\end{array}\right]
\nonumber
\label{gcm3}
\eeq \nod for $a \ge 4$, it is a priori possible that the 3--dimensional unipotent group $U_{sts}$ is a Heisenberg group (see \cite{AR} for details on $\Phi^{+}$ in this case).

\end{document}